\documentclass[11pt]{amsart}

\usepackage{amscd,amsmath,amssymb,amsfonts,amsthm,mathrsfs,amscd,url}
\usepackage[a4paper,text={160mm,240mm},centering,headsep=5mm,footskip=10mm]{geometry}
\usepackage[all]{xy}

\usepackage[square,sort,comma,numbers]{natbib}
\usepackage{stmaryrd}
\usepackage{array}
\usepackage{color}

\usepackage{datetime}
 \usepackage[ngerman,english]{babel}
 
\usepackage[utf8]{inputenc}

\usepackage{graphicx}
\usepackage[all]{xy}
\usepackage{mathrsfs}

\usepackage{setspace}
\DeclareMathOperator{\Hom}{\mathscr{H}\text{\kern -3pt {\calligra\large om}}\,}
\DeclareMathOperator{\Ext}{\mathscr{E}\text{\kern -3pt {\calligra\large xt}}\,}
\def\Spec{{\rm Spec}}

\usepackage[square,sort,comma,numbers]{natbib}
\usepackage{stmaryrd}
\usepackage{array}
\usepackage{color}

\numberwithin{equation}{section}

\newtheorem{theorem}{Theorem}[subsection]
\newtheorem{definition}[theorem]{Definition}

\newtheorem{proposition}[theorem]{Proposition}

\newtheorem{corollary}[theorem]{Corollary}

\newtheorem{lemma}[theorem]{Lemma}
\newtheorem{theoremalpha}{Theorem}

\newtheorem{propalpha}[theoremalpha]{Proposition}
\theoremstyle{remark}
\newtheorem{remark}[theorem]{Remark}

\def\F{\mathbb{F}}
\def\Ome{\Omega}

\def\Omek{\Omega_{X/k}}

\def\Q{\mathbb{Q}}
\def\Z{\mathbb{Z}}
\def\Zpn{\Z/p^n\Z}
\def\Zp1{\Z/p\Z}
\def\Fp{\mathbb{F}_q}
\def\Wno{W_n\Ome}
\def\m{\mathfrak{m}}
\def\Hom{\mathrm{Hom}}
\def\sHom{{\mathcal H}om}

\begin{document}

\title[Duality for relative logarithmic de Rham-Witt sheaves]
{Duality for relative logarithmic de Rham-Witt sheaves on semistable schemes over $\mathbb{F}_q[[t]]$}
\author{ Yigeng Zhao}
\address{Fakult\"at f\"ur Mathematik\\ Universit\"at Regensburg \\ 93040 Regensburg\\ Germany}
\email{yigeng.zhao@mathematik.uni-regensburg.de}
\thanks{The author is supported by the DFG through CRC 1085 \emph{Higher Invariants} (Universit\"at Regensburg).}

\begin{abstract}
We study duality theorems for the relative logarithmic de Rham-Witt sheaves on semi-stable schemes $X$ over a local ring $\mathbb{F}_q[[t]]$, where $\mathbb{F}_q$ is a finite field. As an application, we obtain a new filtration on the maximal abelian quotient $\pi^{\text{ab}}_1(U)$ of the \'etale fundamental groups $\pi_1(U)$ of an open subscheme $U \subseteq X$,  which gives a measure of ramification along a 
 divisor $D$ with normal crossing and $\text{Supp}(D) \subseteq X-U$. This filtration coincides with the Brylinski-Kato-Matsuda filtration in the relative dimension zero case.

\end{abstract}

\maketitle
\tableofcontents

\section*{Introduction}

The motivation of this paper is to study ramification theory for higher-dimensional schemes of characteristic $p>0$. In the light of class field theory, we want to define a filtration on the abelianized \'etale fundamental group of an open subscheme $U$ of a regular scheme $X$, which measures the ramification of a finite \'etale covering of $U$ along the complement $D=X-U$.  More precisely, let $D=\bigcup_{i=1}^sD_i$ be a reduced effective Cartier divisor on $X$ such that Supp($D$) has simple normal crossing, where $D_1,\cdots, D_s$ are the irreducible components of $D$, and let $U$ be its complement in $X$. We want to define a quotient group $\pi^{\text{ab}}_1(X,mD)/p^n$ of $\pi^{\text{ab}}_1(U)/p^n$, for a divisor $mD=\sum_{i=1}^sm_iD_i$ with each $m_i\geq 1$, which classifies the finite \'etale coverings of degree $p^n$ over $U$ with ramification bounded by $mD$ along the divisor $D$. 

We define the quotient group $\pi^{\text{ab}}_1(X,mD)/p^n$ by using the relationship between $\pi^{\text{ab}}_1(U)/p^n$ and $H^1(U,\Zpn)$, and by then applying a duality theorem for certain cohomology groups.  For this we assume some finiteness conditions on the scheme $X$.  The first is to assume that $X$ is smooth and proper of dimension $d$ over the finite field $\Fp$.  For a finite \'etale covering of $U$ of degree  $\ell^n$, where $\ell$ is a prime different from $p$, this was already done by using duality theory  in $\ell$-adic cohomology [SGA4 \cite{SGA4}]. The Poincar\'e-Pontryagin duality theorem \cite{saito1989} gives isomorphisms 
\[\pi^{\text{ab}}_1(U)/\ell^n \cong \text{Hom}(H^1(U, \Z/{\ell^n}),\Q/\Z) \cong H^{d}_{\text{c}}(U,\Z/{\ell^n}(d)).\]
The case of degree  $p^n$ coverings is more subtle, as we deal with wild ramification and there is no obvious analogue of cohomology with compact support for logarithmic de Rham-Witt sheaves. In \cite{jszduality}, we proposed a new approach. That duality theorem, based on Serre's coherent duality and Milne's duality theorems, together with Pontryagin duality give isomorphisms 
\[ \pi^{ab}_{1}(U)/p^n\cong \text{Hom}(H^1(U, \Zpn),\Q/\Z) \cong \varprojlim_mH^{d}(X, \Wno_{X|mD,\log}^d) , \]
where $\Wno_{X|mD,\log}^d$ (see Definition \ref{twistedlog}) is the relative logarithmic de Rham-Witt sheaf with respect to the divisor $mD$. Using these isomorphisms, we define a quotient $\pi^{\text{ab}}_1(X,mD)/p^n$ of $\pi^{ab}_1(U)/p^n$,  ramified of order $mD$ where $m$ is the smallest value such that the above isomorphism factors through $H^{d}(X, \Wno_{X|mD,\log}^d)$. We may think of $\pi^{\text{ab}}_1(X,mD)/p^n$ as the quotient of  $\pi^{ab}_{1}(U)$ classifying abelian \'etale coverings of $U$ whose degree divides $p^n$ with ramification bounded by $mD$. In  \cite{kerzsaito}\cite{kerzsaitochowgroup}, Kerz and Saito also defined a similar quotient group by using curves on $X$. 

The second finiteness condition is to assume that $X$ is proper (or projective) over a discrete valuation ring $R$. More precisely, we may assume that $X$ is a  proper semi-stable scheme over $\text{Spec}(R)$.  Then there are two cases: mixed and equi-characteristic. In the mixed characteristic case, instead of logarithmic de Rham-Witt sheaves, Sato \cite{satotatetwist} defined the $p$-adic Tate twists, and proved an arithmetic duality theorem for $X$. In  \cite{uzun}, Uzun proved that over the $p$-adic field $\pi^{ab}_1(U)/n$ is isomorphic to some motivic homology groups, for all $n>0$.

In this paper, we treat the equi-characteristic case, where the wildly ramified case has not been considered before. We follow the approach suggested by Jannsen and Saito in \cite{jszduality}. 
The main result of this paper is the following theorem.
\begin{theoremalpha}[see Theorem \ref{main theorem}] \label{theorema}
Let $X \to \text{Spec}(\Fp[[t]])$ be a projective strictly semistable scheme of relative dimension $d$, and let $X_s$ be its special fiber. Let $D$ be an effective Cartier divisor on $X$ such that Supp$(D)$ has simple normal crossing, and let $U$ be its open complement. Then there is a perfect pairing of topological $\Zpn$-modules
\[	H^i(U, \Wno_{U,\log}^r) \times \varprojlim\limits_{m} H^{d+2-i}_{X_s}(X, \Wno_{X|mD,\log}^{d+1-r}) \to H^{d+2}_{X_s}(X, \Wno_{X,\log}^{d+1}) \xrightarrow{\text{Tr}} \Zpn,  \]
where the $H^{d+2-i}_{X_s}(X, \Wno_{X|mD,\log}^{d+1-r})$ is endows with the discrete topology, and $H^i(U, \Wno_{U,\log}^r)$ is endowed with the direct limit topology of compact-open groups.
\end{theoremalpha}
Therefore, we can define a filtration $\text{Fil}_{\bullet}$ on $H^i(U, \Wno_{U,\log}^r) $ via the inverse limit (see Definition \ref{newfiltration}).
This theorem and Pontryagin duality give isomorphisms
\[ \pi^{ab}_{1}(U)/p^n \cong \text{Hom}(H^1(U, \Zpn), \Q/\Z) \cong \varprojlim_mH_{X_s}^{d+1}(X, \Wno_{X|mD,\log}^{d+1}), \]
and so we may define  $\pi^{\text{ab}}_1(X, mD)/p^n$ as the dual of $\text{Fil}_mH^1(U, \Zpn)$ (see Definition \ref{newfiltration}). 

This paper is organized as follows.

In the first section,  we will prove a new purity theorem on certain regular schemes.  Its cohomological version will be used later for the trace map in the above duality theorem. 
\begin{theoremalpha}[see Theorem \ref{newpurity} ]
Assume $X$ is as before, and $i: X_s \hookrightarrow X$ is the special fiber, which is a reduced divisor  and has simple normal crossing. Then there is a canonical isomorphism
\[ Gys_{i,n}^{\log}:  \nu_{n,X_s}^d[-1] \xrightarrow{\quad \cong \quad} Ri^!\Wno_{X,\log}^{d+1}  \]
in $D^{+}(X_s, \Zpn)$, where $\nu_{n, X_s}^d=\ker(\bigoplus_{x\in X_s^0}i_{x*}\Wno_{x,\log}^d\to \bigoplus_{x\in X_s^1}i_{x*}\Wno_{x,\log}^{d-1} )$ and $X_s^i$ is the set of codimension $i$ points on $X_s$, for $i=0,1$.
\end{theoremalpha}

Our goal in the second section is to develop an absolute coherent duality on $X$. This can be achieved by combining an absolute coherent duality on the local ring $B=\text{Spec}(\Fp[[t]])$ and a relative duality for $f$. For the former, we use the Grothendieck local duality, and the latter is following theorem. 

\begin{theoremalpha}[see Theorem \ref{fc}]
Let $f: X \to B=\text{Spec}(\Fp[[t]])$ be a projective strictly semistable scheme. Then there is a canonical trace isomorphism
$$\text{Tr}_f : \Ome_X^{d+1}[d]  \xrightarrow{\cong}   f^!\Ome_B^1.$$
\end{theoremalpha}

In the third section, we study the duality theorems of logarithmic de Rham-Witt sheaves on our projective semistable scheme. In fact, we will prove two duality theorems. The first one is for $H^i(X, \Wno^j_{X,\log})$, which we call unramified duality.
\begin{theoremalpha}[see Theorem \ref{unramifiedduality}] The natural pairing
\[ H^i(X, \Wno^j_{X,\log})\times H^{d+2-i}_{X_s}(X, \Wno^{d+1-j}_{X,\log}) \rightarrow H^{d+2}_{X_s}(X, \Wno^{d+1}_{X,\log})\xrightarrow{Tr} \Zpn\]
induces an isomorphism
\[  H^i(X, \Wno^j_{X,\log}) \xrightarrow{\cong} Hom_{\Zpn}(H^{d+2-i}_{X_s}(X, \Wno^{d+1-j}_{X,\log}),\Zpn). \]
of $\Zpn$-modules.
\end{theoremalpha}
The second is the above main Theorem \ref{theorema} for $H^i(U, \Wno^j_{U,\log})$. We call it ramified duality. To define the pairing, we do further studies on the sheaves $\Wno_{X|mD,\log}^r$ in  the middle two subsections.

In the last section, we will compare our new filtration with previously known filtrations in some special cases. The first interesting case would be the filtration  in local ramification theory. We can show that for the local field $K=\Fp((t))$ our filtration agree with the non-log version of Brylinski-Kato filtration $\text{fil}_{\bullet}H^1(K, \Z/p^n\Z)$ \cite{brylinski} \cite{katoswanconductor} defined by Matsuda \cite{matsuda}:
\begin{propalpha}[Proposition \ref{filagree}]
For any integer $m \geq 1$, we have $\text{Fil}_mH^1(K, \Z/p^n\Z)=\text{fil}_mH^1(K, \Z/p^n\Z)$.
\end{propalpha}
\bigskip
\textbf{Acknowledgments}
The author would like to thank heartily Uwe Jannsen for his advice and support over
the past few years, and thank Shuji Saito for his insightful discussions and encouragement. The author also want to thank Moritz Kerz for numerous constructive comments, to Georg Tamme, Patrick Forr\'e,  Florian Strunk and Yitao Wu for helpful conversations, and to the anonymous Referee for her/his valuable comments.
\section{Purity}
\subsection{Logarithmic de Rham-Witt sheaves}
Let $X$ be a scheme of dimension $d$ over  a perfect field $k$ of characteristic $p>0$, and let $W_n(k)$ be the ring of  Witt vectors of length $n$. 

Based on ideas of Lubkin, Bloch and Deligne, Illusie  defined the de Rham-Witt complex \cite{illusiederham}. Recall the de Rham-Witt complex $W\Omek^{\bullet}$ is the inverse limit of an  inverse system $(\Wno_{X/k}^{\bullet})_{n\geq 1}$ of complexes
\[ \Wno_{X/k}^{\bullet}:=(\Wno_{X/k}^{0} \xrightarrow{d} \Wno_{X/k}^{1} \rightarrow \cdots  \xrightarrow{d} \Wno_{X/k}^{i} \xrightarrow{d} \cdots ) \]
of sheaves of $W_n\mathcal{O}_X$-modules on the Zariski site of $X$. The complex $\Wno_{X/k}^{\bullet}$ is called the de Rham-Witt complex of level $n$.

This complex $\Wno_{X/k}^{\bullet}$ is a strictly anti-commutative differential graded $W_n(k)$-algebra. In the rest of this section, we will omit the subscript $/k$ to simplify the notation.

 We have the following operators on the de Rham-Witt complex (\cite[I]{illusiederham}):
\begin{itemize}
\item[(i)] The projection $ R: \Wno_X^{\bullet} \to W_{n-1}\Omega_X^{\bullet}$, which is a surjective homomorphism of differential graded algebras.
 \item[(ii)] The Verschiebung  $V: \Wno_X^{\bullet} \to W_{n+1}\Omega_X^{\bullet}$ , which is an additive homomorphism.
 \item[(iii)] The Frobenius  $F: \Wno_X^{\bullet} \to W_{n-1}\Omega_X^{\bullet} $, which is a homomorphism of differential graded algebras.
\end{itemize}

\begin{proposition}[{\cite[I 1.13,1.14]{illusiederham}}]
\item[(i)] For each $n\geq 1$, and each $i$, $\Wno_X^i$ is a quasi-coherent $W_n\mathcal{O}_X$-module.
\item[(ii)] For any \'etale morphism $f: X \to Y$, $f^*\Wno_Y^i \to \Wno_X^i$ is an isomorphism of  $W_n\mathcal{O}_X$-modules. 
\end{proposition}

\begin{remark}
Let $\mathscr{F}$ be a quasi-coherent on $X$, we denote its associated sheaf on $X_{\text{\'et}}$ by $\mathscr{F}_{\text{\'et}}$, then we have $H^i(X_{\text{Zar}}, \mathscr{F})=H^i(X_{\text{\'et}},\mathscr{F}_{\text{\'et}})$, for all $i\geq 0$\cite[III 3.7]{milneetale}. By the above proposition, we may also denote $\Wno_X^i$ as sheaf on $X_{\text{\'et}}$,
and its \'etale and Zariski cohomology groups are agree.
\end{remark}

Cartier operators are another type of operators on the de Rham-Witt complex.  Before stating the theorem, we set
\[ Z\Wno_{X}^i:=\text{Ker}(d: \Wno_X^i \to \Wno_X^{i+1});\]
\[ B\Wno_{X}^i:=\text{Im}(d: \Wno_X^{i-1} \to \Wno_X^{i});\]
\[ \mathscr{H}^i(\Wno_X^{\bullet}):=Z\Wno_{X}^i/B\Wno_{X}^i ;\]
\begin{equation*}
\begin{array}{rl}
Z_1W_{n}\Ome_X^i:&=\text{Im}(F: W_{n+1}\Ome_X^i \to \Wno_X^i)\\
&=\text{Ker}(F^{n-1}d: \Wno_X^i \xrightarrow{d} \Wno_X^{i+1} \xrightarrow{F^{n-1}} \Ome_X^{i+1}).
\end{array}
\end{equation*}
Since $W_1\Omega_X^i \cong \Omega_X^i$, $ZW_1\Omega_X^i$(resp. $BW_1\Omega_X^i$) is also denoted by $Z\Omega_X^i $(resp. $ B\Omega_X^i$). Note that $Z\Omega_X^i$, $B\Omega_X^i$, and $\mathscr{H}^i(\Omega_X^{\bullet})$ can be given $\mathcal{O}_X$-module structures via the absolute Frobenius morphism $F$ on $\mathcal{O}_X$. 
\begin{theorem}[{Cartier, \cite[Thm. 7.2]{katz}, \cite[Thm. 3.5]{illusiefrobenius}}] \label{catieroperator}
Suppose X is of finite type over $k$. Then there exists a unique $p$-linear homomorphism of graded $\mathcal{O}_X$-algebras
\[ C^{-1}:  \bigoplus \Omega_X^i\longrightarrow    \bigoplus\mathscr{H}^i(\Omega_X^{\bullet}) \]
satisfying the following two conditions:
\begin{itemize}
\item[(i)] For $a \in \mathcal{O}_X$, $C^{-1}(a)=a^p$;
\item[(ii)] For $dx\in \Omega_X^1$, $C^{-1}(dx)=x^{p-1}dx$.
\end{itemize}
If $X$ is moreover smooth over $k$, then $C^{-1}$ is an isomorphism. It is called inverse Cartier isomorphism. The inverse of $C^{-1}$ is called Cartier operator, and is denoted by $C$. 
\end{theorem}
Higher Cartier operators can be defined as follows, which comes back to the above theorem in the case $n=1$.
\begin{proposition}[{\cite[III]{illusieraynaud},  \cite[$\S$ 4]{katoI}}]\label{highercartier}
If $X$ is smooth over $k$, then there is a unique higher Cartier morphism $C: Z_1W_{n+1}\Ome_X^i \to \Wno_X^i$ such that the diagram
\[ \xymatrix{ Z_1W_{n}\Ome_X^i \ar[dr]_C \ar[rr]^V && W_{n+1}\Ome_X^q\\
&\Wno_X^i \ar[ur]_p&
 }\]
 is commutative. We have an isomorphism \[ \Wno_X^i \xrightarrow[\cong]{F} Z_1W_{n+1}\Ome_X^i/dV^{n-1}\Ome_X^{i-1},\]
 and an exact sequence
 \[   0 \to dV^{n-1}\Ome_X^{i-1} \to  Z_1W_{n+1}\Ome_{X}^i \xrightarrow{C} \Wno_X^i.\]

\end{proposition}

For a smooth variety $X$ over $k$,  the composite morphism
\[ W_{n+1}\Ome_X^{i} \xrightarrow{F} \Wno_X^{i} \twoheadrightarrow \Wno_X^{i}/dV^{n-1}\Ome_X^{i-1}\] 
is trivial on $\text{Ker}(R:W_{n+1}\Ome_X^{i} \to \Wno_X^{i} )=V^n\Ome_X^i+dV^n\Ome_X^{i-1}$. Therefore $F$ induces a morphism $$ F: \Wno_X^{i} \to \Wno_X^{i}/dV^{n-1}\Ome_X^{i-1}. $$
\begin{definition}
Let $X$ be a smooth variety over $k$. For any positive integer $n$, and any non-negative integer $i$, we define the $i$-th logarithmic de Rham-Witt sheaf of length $n$ as
\[ \Wno_{X,\log}^{i}:=\text{Ker}(\Wno_X^i\xrightarrow{1-F} \Wno_X^{i}/dV^{n-1}\Ome_X^{i-1}).\]
For any $x\in X$, we denote $\Wno_{x,\log}^i:=\Wno_{\kappa(x),\log}^i$, where $\kappa(x)$ is the residue field at $x$.
\end{definition}

\begin{remark}(Local description of $\Wno_{X,\log}^{i}$  \cite[I 1.3]{illusiederham})\label{localdescriptionoflog}
The $i$-th logarithmic de Rham-Witt sheaf $\Wno_{X,\log}^{i}$ is the additive subsheaf of  $\Wno_{X}^i $, which is  \'etale locally generated by sections $d\log[x_1]_n\cdots d\log[x_i]_n$, where $x_i\in  \mathcal{O}_X^{\times}$, $[x]_n$ is the Teichm\"uller representative of $x$ in $W_n\mathcal{O}_X$, and $d\log[x]_n:=\frac{d[x]_n}{[x]_n}$. In other words, it is generated by the image of 
\[ \begin{array}{crll}
        d\log : & (\mathcal{O}_X^{\times})^{\otimes i} & \longrightarrow & \Wno_X^{i} \\\relax
         & (x_1,\cdots, x_i) & \longmapsto & d\log[x_1]_n\cdots d\log[x_i]_n 
\end{array} \]

\end{remark}
\begin{proposition}[\cite{colliottorsion},\cite{grossuwa2},\cite{illusiederham}] \label{exactseqsmooth}
For a smooth variety $X$ over $k$, we have the following exact sequences of \'etale sheaves on $X$:
\begin{itemize}
\item[(i)] $0 \to \Wno_{X,\log}^i \xrightarrow{p^m} W_{n+m}\Ome_{X,\log}^i \xrightarrow{R} W_m\Ome_{X,\log}^i \to 0$;
\item[(ii)] $0 \to \Wno_{X,\log}^i \to W_{n}\Ome_{X}^i \xrightarrow{1-F}  W_n\Ome_{X}^i/dV^{n-1}\Ome_{X}^{i-1} \to 0$;
\item[(iii)] $0 \to \Wno_{X,\log}^i \to Z_1W_{n}\Ome_{X}^i \xrightarrow{C-1} W_n\Ome_{X}^i \to 0$.
\end{itemize}
\end{proposition}
\begin{proof}
The first assertion is Lemma 3 in \cite{colliottorsion},  and the second is Lemma 2 in loc.cit.. The last one is Lemma 1.6 in \cite{grossuwa2}, which can easily be deduce from (ii). In particular, for $n=1$, (iii) can be also found in \cite{illusiederham}.
\end{proof}

The logarithmic de Rham-Witt sheaves $\Wno_{X,\log}^i$ are $\Zpn$-sheaves, which  have a similar duality theory as the $\Z/\ell^n\Z$-sheaves $\mu_{\ell}^{\otimes n}$ with $\ell\neq p$ for a smooth proper variety:
\begin{theorem}(Milne duality \cite[1.12]{milneduality})\label{milneduality}
Let $X$ be a smooth proper variety over $k$ of dimension $d$, and let $n$ be a positive integer. Then the following holds:
\begin{itemize}
\item[(i)] There is a canonical trace map $\text{tr}_X:H^{d+1}(X, \Wno_{X,\log}^d) \to\Zpn$. It is bijective if $X$ is connected;
\item[(ii)] For any integers $i$ and $r$ with $0\leq r\leq d$, the natural pairing 
\[ H^i(X, \Wno_{X,\log}^r) \times  H^{d+1-i}(X, \Wno_{X,\log}^{d-r}) \to \Zpn \]
is a non-degenerate pairing of finite $\Zpn$-modules.
\end{itemize}
\end{theorem}
\begin{remark}
The proof  can be obtained in the following way: using the exact sequence (i) in Proposition \ref{exactseqsmooth}, we reduce to the case $n=1$,  which can be obtained  from Serre's coherent duality via the exact sequence (ii) and (iii) in the same proposition.
\end{remark}

\subsection{Normal crossing varieties}
In \cite{satoncv}, Sato generalized the definition of logarithmic de Rham-Witt sheaves from smooth varieties to more general varieties, and proved that they share  similar properties on normal crossing varieties.
 
 Let $Z$ be a variety over $k$ of dimension $d$. For a non-negative integer $m$ and a positive integer $n>0$,  we denote by $C^{\bullet}_n(Z,m)$ the following complex of \'etale sheaves on X 
 \[ \bigoplus_{x\in Z^0}i_{x*}\Wno_{x,\log}^m \xrightarrow{(-1)^m\cdot \partial} \bigoplus_{x\in Z^1}i_{x*}\Wno_{x,\log}^{m-1}\xrightarrow{(-1)^m\cdot \partial} \cdots  \xrightarrow{(-1)^m\cdot \partial} \bigoplus_{x\in Z^c}i_{x*}\Wno_{x,\log}^{m-c} \rightarrow \cdots \]
 where $i_x$ is the natural map $x \to Z$, $Z^i$ is the set of codimension $i$ points, and $\partial$ denotes the sum of Kato's residue maps (see \cite[1.7]{jssduality}\cite{katoresiduemap}).
 \begin{definition}
 \begin{itemize}
 \item[(i)] The $m$th homological logarithmic Hodge-Witt sheaf is defined as the 0-th cohomology sheaf $\mathscr{H}^0(C^{\bullet}_n(Z,m)) $ of the complex $C^{\bullet}_n(Z,m)$, and denoted by $\nu_{n,Z}^m$.
 \item[(ii)] The $m$th cohomological Hodge-Witt sheaf is the image of 
  \[ d\log:  (\mathcal{O}_Z^{\times})^{\otimes m} \longrightarrow \ \bigoplus_{x\in Z^0}i_{x*}\Wno_{x,\log}^m, \] and denoted by $\lambda_{n,Z}^m$.
 \end{itemize}
  
 \end{definition}
 \begin{remark}\label{smoothlog}
 If $Z$ is smooth, then $\nu_{n,Z}^m=\lambda_{n,Z}^m=\Wno_{Z,\log}^m$, but in general $ \lambda_{n,Z}^m \subsetneq \nu_{n,Z}^m $ \cite[Rmk. 4.2.3]{satoncv}.
 \end{remark}
 
 \begin{definition}
 The variety $Z$ is called normal crossing variety if it is everywhere \'etale locally isomorphic to \[ Spec(k[x_0,\cdots,x_d]/(x_0\cdots x_a) )\] for some integer $a\in [0,d]$, where $d=\text{dim}(Z)$.  A normal crossing variety is called simple if every irreducible component is smooth.
 \end{definition}
 \begin{proposition}(\cite[Cor. 2.2.5(1)]{satoncv}) \label{acyclic}
 For a normal crossing variety $Z$, the natural map $ \nu_{n,Z}^m \longrightarrow C^{\bullet}_n(Z,m) $ is a quasi-isomorphism of complexes.
 \end{proposition}
 
\begin{theorem}(\cite[Thm. 1.2.2]{satoncv})\label{satoduality}
Let $Z$ be a normal crossing variety over  a finite field, and proper of dimension $d$. Then the following holds: 
\begin{itemize}
\item[(i)] There is a canonical trace map $\text{tr}_Z: H^{d+1}(Z, \nu_{n,Z}^d) \to \Zpn$. It is bijective if $Z$ is connected.
\item[(ii)] For any integers $i$ and $j$ with $0\leq j \leq d$, the natural pairing 
\[ H^i(Z, \lambda_{n,Z}^j) \times H^{d+1-i}(Z, \nu_{n,Z}^{d-j})\rightarrow H^{d+1}(Z, \nu_{n,Z}^d) \xrightarrow{tr_Z} \Zpn  \]
is a non-degenerate pairing of finite $\Zpn$-modules.
\end{itemize}
\end{theorem}

\subsection{Review on purity}

In the $\ell$-adic setting, the sheaf $\mu_{\ell^n}^{\otimes r}$ on a regular scheme has purity. This was called Grothendieck's absolute purity conjecture, and it was proved by Gabber and can be found in  \cite{fujiwara}. In the $p$-adic case, we may ask if purity holds for the logarithmic de Rham-Witt sheaves. But these sheaves only have semi-purity(see Remark \ref{semipurity} below).
\begin{proposition}(\cite{groschernclass})\label{vanishingcodimension}
	Let $i: Z  \hookrightarrow X$ be a closed immersion of smooth schemes of codimension $c$ over a perfect field $k$ of characteristic $p>0$. Then, for $r\geq 0$ and $ n \geq 1$, $R^mi^!\Wno_{X,\log}^r=0$ if $m\neq c, c+1$.
\end{proposition}
For the logarithmic de Rham-Witt sheaf at top degree, i.e., $\Wno_{X, \log}^d$ where $d=\text{dim}(X)$, the following theorem tells us  $R^{c+1}i^!\Wno_{X,\log}^d=0$. 
\begin{theorem}(\cite{grossuwa,milneduality,suwa}) \label{classicalpurity}
Assume $i: Z  \hookrightarrow X$ is as above.  Let $d=\text{dim}(X)$. Then, for $n\geq 1$, there is a canonical isomorphism (called Gysin morphism)
\[  Gys_{i}^d : \Wno_{Z, \log}^{d-c}[-c] \xrightarrow{\quad \cong \quad} Ri^!\Wno_{X,\log}^d \]
in  $D^+(Z, \Zpn)$.
\end{theorem}
\begin{remark}\label{semipurity}
Note that the above theorem is only for the $d$-th logarithmic de Rham-Witt sheaf. For $m< d$,  the $R^{c+1}i^!\Wno_{X,\log}^m$ is non zero in general \cite[Rem. 2.4]{milneduality}. That's the reason why we say they only have semi-purity.
\end{remark}

Sato generalized the above theorem to normal crossing varieties.
\begin{theorem}(\cite[Thm. 2.4.2]{satoncv})\label{satopurity}
Let $X$ be a normal crossing varieties of dimension $d$, and $i: Z \hookrightarrow X$ be a closed immersion of pure codimension $c\geq 0$. Then, for $n\geq 1$, there is a canonical isomorphism(also called Gysin morphism)
\[ Gys_i^d: \nu_{n,Z}^{d-c}[-c] \xrightarrow{\quad \cong \quad}   Ri^!\nu_{n,X}^{d} \]
in $D^+(Z, \Zpn)$.
\end{theorem}
\begin{remark}
The second Gysin morphism coincides with the first one, when $X$ and $Z$ are smooth \cite[2.3,2.4]{satoncv}. In loc.cit., Sato studied the Gysin morphism of $\nu_{n,X}^r$ for $0\leq r \leq d$. In fact the above isomorphism for $\nu_{n,X}^d$ was already proved by Suwa \cite[2.2]{suwa} and Morse \cite[2.4]{moser}.

\end{remark}

\begin{corollary}
If $i: Z \hookrightarrow X $ is a normal crossing divisor(i.e normal crossing subvariety of codimension 1 in $X$) and $X$ is smooth,  then we have 
\[ Gys_i^d: \nu_{n,Z}^{d-1}[-1] \xrightarrow{\quad \cong \quad} Ri^!\Wno_{X,\log}^d \]
in $D^+(Z, \Zpn)$.
\end{corollary}
\begin{proof}
This follows from the fact that $\Wno_{X,\log}^d = \nu_{n,X}^d $, when $X$ is smooth over $k$.
\end{proof}

We want to generalize this corollary to the case where $X$ is regular. For this, we need a purity result of Shiho  \cite{shihopurity}, which is a generalization of Theorem \ref{classicalpurity}  for smooth schemes to regular schemes.

\begin{definition}\label{definitionOfLog}
Let $X$ be a scheme over $\F_p$, and $i \in \mathbb{N}_0, n \in \mathbb{N}$. Then we define the $i$-th logarithmic de Rham-Witt sheaf $\Wno_{X,\log}^i$ as the subsheaf of $\Wno_{X}^i$, which is  generated by the image of 
\[ d\log: (\mathcal{O}_X^{\times})^{\otimes i} \longrightarrow \Wno_{X}^i,  \]
where $d\log$ is defined by \[  d\log(x_1\otimes \cdots \otimes x_i)=d\log[x_1]_n\cdots d\log[x_i]_n, \]
and $[x]_n$ is the Teichm\"uller representative of $x$ in $W_n\mathcal{O}_X$.

\end{definition}

\begin{remark}
Note that this definition is a simple generalization of the classical definition for smooth $X$, by comparing with the local description of logarithmic de Rham-Witt sheaves in Remark \ref{localdescriptionoflog}.
\end{remark}

As in  Theorem \ref{catieroperator}, we can define the inverse Cartier operator similarly for a scheme over $\F_p$.  Using the N\'eron-Popescu approximation theorem \cite{swanneron}(see Theorem \ref{popescu} below) and  Grothendieck's limit theorem (SGA 4 \cite[VII, Thm. 5.7]{SGA4}), Shiho showed the following results.

\begin{proposition} (\cite[Prop. 2.5]{shihopurity} )
If $X$ is a regular scheme over $\F_p$, the  inverse Cartier homomorphism $C^{-1}$ is an isomorphism. 
\end{proposition}
Using the same method, we can prove:
\begin{theorem}
The results of Proposition \ref{highercartier} also hold for a regular scheme over $\F_p$.
\end{theorem}
\begin{proposition} (\cite[Prop. 2.8, 2.10, 2.12]{shihopurity}) \label{logtocoh}
Let $X$ be a regular scheme over $\F_p$. Then we have the following exact sequences:
\begin{itemize}
\item[(i)] $0 \to \Wno_{X,\log}^i \xrightarrow{p^m} W_{n+m}\Ome_{X,\log}^i \xrightarrow{R} W_m\Ome_{X,\log}^i \to 0$;
\item[(ii)] $0 \to \Wno_{X,\log}^i \to W_{n}\Ome_{X}^i \xrightarrow{1-F}  W_n\Ome_{X}^i/dV^{n-1}\Ome_{X}^{i-1} \to 0$;
\item[(iii)] $0 \to \Wno_{X,\log}^i \to Z_1W_{n}\Ome_{X}^i \xrightarrow{C-1} W_n\Ome_{X}^i \to 0$.
\end{itemize}
\end{proposition}
\begin{proof}
The claim (iii) is easily obtained from (ii), as in the smooth case. When $n=1$, this is Proposition 2.10 in loc.cit..
\end{proof}

Let $\mathscr{C}$ be a category of regular schemes of characteristic $p>0$, such that, for any $x\in X$, the absolute Frobenius $\mathcal{O}_{X,x} \to \mathcal{O}_{X,x} $ of the local ring $\mathcal{O}_{X,x} $ is finite. Shiho showed the following cohomological purity result.
\begin{theorem}(\cite[Thm. 3.2]{shihopurity}) \label{shihopurity}
Let $X$ be a regular scheme over $\F_p$, and let $i: Z \hookrightarrow X$ be a regular closed immersion of codimension $c$. Assume moreover that $[\kappa(x): \kappa(x)^p]=p^N$ for any $x \in X^0$, where $\kappa(x)$ is the residue field at $x$. Then there exists a canonical isomorphism 
\[ \theta_{i,n}^{q,m,\log}: H^{q}(Z, \Wno_{Z,\log}^{m-c}) \xrightarrow{\quad \cong \quad} H^{q+c}_Z(X, \Wno_{X,\log}^{m}) \]
if $q=0$ holds or if $q>0, m=N, X \in \text{ob}(\mathscr{C})$ hold.
\end{theorem}
\begin{remark}\label{regularvanishingcodimension}
	In \cite[Cor. 3.4]{shihopurity}, Shiho also generalized Proposition \ref{vanishingcodimension} to the case that $Z \hookrightarrow X$ is a regular closed immersion, and without the assumption on the residue fields.
\end{remark}

\begin{corollary}\label{corofshihopurity}
Let $X$ be as in Theorem \ref{shihopurity}, $i_x: x \to X$ be a point of codimension $c$. Then the canonical morphism
\[ \theta_{i_x,n}^{\log}: \Wno_{x,\log}^{N-c}[-c] \xrightarrow{\quad \cong \quad }  Ri_x^!\Wno_{X,\log}^N\]
is an isomorphism in $D^+(x, \Zpn)$.
\end{corollary}
\begin{proof}
Let $X_x$ be the localization of $X$ at $x$. The assertion is a local problem, hence we may assume $X=X_x$.  By the above Remark \ref{regularvanishingcodimension}, we have $R^{j}i_x^!\Wno_{X,\log}^N =0$ for $j \neq c, c+1$.  The natural map $\Wno_{x,\log}^{N-c}[-c] \to R^{c}i_x^!\Wno_{X,\log}^N[0]$ induces the desired morphism $\theta_{i_x,n}^{\log}$, and the above theorem tells us this morphism induces isomorphisms on cohomology groups. An alternative way is to show  $R^{c+1}i_x^!\Wno_{X,\log}^N=0$ directly as Shiho's arguments in the proof of Theorem \ref{shihopurity}.
\end{proof}

\subsection{A new result on purity for semistable schemes}
We recall the following definitions.
\begin{definition}
For a regular scheme $X$ and a divisor $D$ on $X$,  we say that $D$ has normal crossing if it satisfies the following conditions:
\begin{itemize}
\item[(i)] $D$ is reduced, i.e. $D= \bigcup_{i\in I}D_i$ (scheme-theoretically), where $\{ D_i \}_{i\in I}$ is the family of irreducible components of $D$;
\item[(ii)] For any non-empty subset $J \subset I$, the (scheme-theoretically) intersection $\bigcap_{j\in J}D_j$ is a regular scheme of codimension  \#$J$ in $X$, or otherwise empty.
\end{itemize}
If moreover each $D_i$ is regular, we called $D$ has simple normal crossing.
\end{definition}

Let $R$ be a complete discrete valuation ring, with quotient field $K$, residue field $k$, and the maximal ideal $\m=(\pi) $, where $\pi$ is a uniformizer of $R$. 
\begin{definition}
Let $X \to Spec(R)$ be a scheme of finite type over $Spec(R)$. We call $X$ a semistable (resp. strictly semistable) scheme over $Spec(R)$, if it satisfies the following conditions:
\begin{itemize}
\item[(i)] $X$ is regular, $X \to Spec(R)$ is flat, and the generic fiber $X_{\eta}:=X_K:=X\times_{Spec(R)}Spec(K)$ is smooth;
\item[(ii)] The special fiber $X_s:=X_k:=X\times_{Spec(R)}Spec(k)$ is a divisor with normal crossings (resp. simple normal crossings) on $X$.
\end{itemize}
\end{definition}
\begin{remark}[Local description of semistable schemes]
Let $X$ be a semistable scheme over $Spec(R)$, then it is everywhere \'etale locally isomorphic to 
\[  Spec(R[T_0, \cdots, T_d]/(T_0\cdots T_a-\pi))  \]
for some integer $a$ with $a \in [0, d]$, where $d$ denotes the relative dimension of $X$ over $Spec(R)$. 
In particular, this implies the special fiber of a semistable(resp. strictly semistable) scheme is a normal (resp. simple normal) crossing variety.
\end{remark}
Let $k$ be a perfect field of characteristic $p>0$, and let $B:=Spec(k[[t]])$ be the affine scheme given by the formal power series with residue field $k$. Our new purity result is the following theorem.
\begin{theorem} \label{newpurity}
Let $X\to B$ be a projective strictly semistable scheme of relative dimension $d$, and let $i: X_s \hookrightarrow X$ be the natural morphism.  Then, there is a canonical isomorphism
\[ Gys_{i,n}^{\log}:  \nu_{n,X_s}^d[-1] \xrightarrow{\quad \cong \quad} Ri^!\Wno_{X,\log}^{d+1}  \]
in $D^{+}(X_s, \Zpn)$.
\end{theorem}

 We will use Shiho's cohomological purity result (Theorem \ref{shihopurity}) in the proof, and the following lemma guarantees our $X$ satisfies the assumption there.

\begin{lemma} Let $X$ be as in Theorem \ref{newpurity}.
\item[(i)] Let $A$ be a ring of characteristic $p>0$. If the absolute Frobenius $F: A \to A,  a \mapsto a^p$ is finite, then the same holds for any quotient or localization. 
\item[(ii)] For any $x \in X$, the absolute Frobenius $F: \mathcal{O}_{X,x} \to \mathcal{O}_{X,x}$ is finite. In particular, our $X$ is in the category $\mathscr{C}$.
\item[(iii)] For any $x \in X^0$,  we have $[\kappa(x): \kappa(x)^p]=p^{d+1}$.
\end{lemma}
\begin{proof}
(i) By the assumption, there is a surjection of $A$-modules $\bigoplus_{i=1}^m A \twoheadrightarrow A$ for some $m$, where the $A$-module structure in the target is twisted by $F$. For the quotienting out by an ideal $I$, then tensoring with $A/I$, we still have a surjection $\bigoplus_{i=1}^m A/I \twoheadrightarrow A/I $ . If $S$ is a multiplicative set, then tensoring with $S^{-1}A$ still gives a surjection  $\bigoplus_{i=1}^m S^{-1}A \twoheadrightarrow S^{-1}A $ .

(ii) Note that $k[[t]]$ as $k[[t]]^p$-module is free with basis $\{1, t, \cdots, t^{p-1} \}$. The absolute Frobenius on the polynomial ring $k[[t]][x_1,\cdots,x_n]$ is also finite. Now the local ring $\mathcal{O}_{X,x}$ is obtained from a polynomial ring over $k[[t]]$ after passing to a quotient and a localization. Hence the assertion follows by (i).

(iii) For $x\in X^0$, the transcendence degree $\text{tr.deg}_{k((t))}\kappa(x)=d$, and $\kappa(x)$ is a finitely generated extension over $k((t))$. So the $p$-rank of $\kappa(x)$ is the $p$-rank of $k((t))$ increased by $d$, and we know that $[k((t)):k((t))^p]=p$.
\end{proof}
We need the following result of Moser.
\begin{proposition}(\cite[Prop. 2.3]{moser})\label{moservanishing}
Let $Y$ be a scheme of finite type over a perfect field of characteristic $p>0$, let $i_x: x \to Y$ be a point of $Y$ with  dim($\overline{\{x\}}$)$=c$. Then we have $R^qi_{x*}\Wno_{x,\log}^c=0$, for all $n\geq 1$ and $q\geq 1$.
\end{proposition}
With the help of the above preparations, we can now prove the theorem. 
\begin{proof}[Proof of Theorem \ref{newpurity}]\renewcommand{\qedsymbol}{}
By  \cite[1.5]{jssduality},  we have the following local-global spectral sequence of \'etele sheaves on $Z:=X_s$:
\[ E^{u,v}_1 = \bigoplus_{x\in Z^u} R^{u+v}\imath_{x*}(Ri_x^!\Wno_{X,\log}^{d+1}) \Longrightarrow R^{u+v}i^!(\Wno_{X,\log}^{d+1}) \]
where, for $x \in Z$, $\imath$ (resp. $i_x=i \circ \imath$) denotes the natural map $x \to Z$ (resp. $x \to X$).

For $x \in Z^u$, we have an isomorphism $$\theta_{i_x,n}^{\log}:  \Wno_{x,\log}^{d-u}[-u-1] \xrightarrow{\quad \cong \quad} Ri_x^!\Wno_{X,\log}^{d+1}, $$ by Corollary \ref{corofshihopurity}.
Then \[ E^{u,v}_1=\bigoplus_{x\in Z^u}R^{v-1}\imath_{x*}\Wno_{x,\log}^{d-u}=0, \qquad \text{if} \quad v\neq 1. \]
where the last equality follows from Proposition \ref{moservanishing}. Hence the local-global spectral sequence degenerates at the $E_1$-page, i.e., $R^ri^!\Wno_{X,\log}^{d+1}$ is the $(r-1)$-th cohomology sheaf of the following complex
$$\xymatrix@C=0.56cm@R=0.5cm{ \bigoplus\limits_{x\in Z^0} \imath_{x*}R^1i_x^!\Wno_{X,\log}^{d+1} \ar[r]^{d^{00}_1}  &\bigoplus\limits_{x\in Z^1} \imath_{x*}R^2i_x^!\Wno_{X,\log}^{d+1} \ar[r] & \cdots\\
   \cdots\ar[r] &\bigoplus\limits_{x\in Z^{r-1}} \imath_{x*}R^ri_x^!\Wno_{X,\log}^{d+1} \ar[r]^-{d^{r-1,0}_1} & \bigoplus\limits_{x\in Z^{r}} \imath_{x*}R^{r+1}i_x^!\Wno_{X,\log}^{d+1} & \\
   \cdots\ar[r] & \bigoplus\limits_{x\in Z^{d}} \imath_{x*}R^{d+1}i_x^!\Wno_{X,\log}^{d+1} .	
}$$
We denote this complex by $B_{n}^{\bullet}(Z,d)$. Then Corollary \ref{corofshihopurity} implies that,  $\theta_{i_x,n}^{\log}  $ gives an isomorphism $C_{n}^{s}(Z,d) \cong B^{s}_{n}(Z,d)$, for any $s$, i.e., a term-wise isomorphism between the complexes $C_{n}^{\bullet}(Z,d)$ and $B_{n}^{\bullet}(Z,d)$. The following theorem will imply that $\theta_{i_x,n}^{\log} $ induces an isomorphism of complexes.
\end{proof}
\begin{theorem}\label{compatibilityofgysinmaps}
Let $X, Z:=X_s$ be as in Theorem \ref{newpurity}. For $x\in Z^c$ and $y\in Z^{c-1}$ with $x \in \overline{\{y\}}$, then the following diagram 
\[ \xymatrix@C=2cm@R=1cm{ H^q(y, \Wno_{y,\log}^{r-c})  \ar[r]^-{(-1)^r\partial_{y,x}^{\text{val}}}\ar[d]^-{\theta_{i_y,n}^{\log}} & H^q(x, \Wno_{x,\log}^{r-c-1}) \ar[d]^-{\theta_{i_x,n}^{\log}}\\
	H^{c+q}_y(X, \Wno_{X,\log}^{r}) \ar[r]_-{\delta_{y,x}^{loc}(\Wno_{X,\log}^r)} & H^{c+1+q}_x(X,\Wno_{X,\log}^r)
	} \]
	commutes if $q=0$ or $ (q,r)=(1,d+1)$, where $\partial_{y,x}^{\text{val}}$ is Kato's residue map and $\delta_{y,x}^{loc}(\Wno_{X,\log}^r)$ is defined in \cite[1.7]{jssduality}.
\end{theorem}

\begin{proof} [Theorem \ref{compatibilityofgysinmaps} implies \ref{newpurity}]
By Theorem \ref{compatibilityofgysinmaps}, under the Gysin isomorphisms $\theta^{\log}$, the morphisms $d^{r,0}_1$ coincide with the boundary maps of logarithmic de Rham-Witt sheaves. Hence $C^{\bullet}_n(Z,d)$ coincides with  $B^{\bullet}_n(Z,d)$. And Proposition \ref{acyclic} said the complex $C^{\bullet}_n(Z,d)$ is acyclic at positive degree, and isomorphic to $\nu_{n,Z}^d$ at zero degree. This shows the claim.
\end{proof}

We are now turning to the proof of commutativity of the above diagram .
\begin{proof}[Proof of Theorem \ref{compatibilityofgysinmaps}]
	We may assume $y\in Z_j$ for some irreducible component $Z_j$ of $Z$.  Note that $Z_j$ is smooth by our assumption. Then we have a commutative diagram:
\[ \xymatrix@C=2.5cm@R=1.5cm{ H^q(y, \Wno_{y,\log}^{r-c}) \ar[r]^-{\theta_{i_y,n}^{\log}} \ar[d]_-{Gys_{i_y}} & H^{c+q}_y(X, \Wno_{X,\log}^{r})\\
	H^{c-1+q}_y(Z_j, \Wno_{Z_j,\log}^{r-1}) \ar[r]^-{\theta_{Z_j\hookrightarrow X}} & H^{c-1+q}_y(Z_j, \mathscr{H}^1_{Z_j}(\Wno_{X,\log}^{r})), \ar[u]	
} \]
where the right-vertical morphism is the edge morphism of Leray spectral sequence.

Hence we have the following diagram:
\[ \xymatrix@C=1cm@R=0.3cm{ H^q(y, \Wno_{y,\log}^{r-c}) \ar[rr]^-{(-1)^r\partial_{y,x}^{\text{val}}}  \ar[dd]_-{Gys_{y \to Z_j}}&&  H^q(x, \Wno_{x,\log}^{r-c-1}) \ar[dd]_-{Gys_{x \to Z_j}}\\
	& (1)& \\
H^{c-1+q}_y(Z_j, \Wno_{Z_j,\log}^{r-1}) 	\ar[rr]^-{\delta_{y,x}^{loc}(\Wno_{Z_j,\log}^{r-1})} \ar[dd]_-{\theta_{Z_j\hookrightarrow X}}  && H^{c+q}_x(Z_j, \Wno_{Z_j,\log}^{r-1}) \ar[dd]_-{\theta_{Z_j\hookrightarrow X}}\\
&(2)&\\
 H^{c-1+q}_y(Z_j, \mathscr{H}^1_{Z_j}(\Wno_{X,\log}^{r})) \ar[rr]^-{\delta_{y,x}^{loc}(\mathscr{H}^1_{Z_j}(\Wno_{X,\log}^{r}))} \ar[dd] && H^{c+q}_x(Z_j, \mathscr{H}^1_{Z_j}(\Wno_{X,\log}^{r}))\ar[dd]\\
 &(3)& \\
 H^{c+q}_y(X, \Wno_{X,\log}^{r}) \ar[rr]^-{\delta_{y,x}^{loc}(\Wno_{X,\log}^{r})} &&H^{c+1+q}_x(X,\Wno_{X,\log}^r).
 }\]
 The square (1) is $(-1)$-commutative if $q=0$ or $(q,r)=(1,d+1)$. This is Theorem 4.1.1 for $q=0$ and Corollary 4.4.1 for  $(q,r)=(1,d+1)$ in \cite{jssduality}. The square (2) is commutative by the functoriality of $\delta_{y,x}^{\text{loc}}$. The square (3) is $(-1)$-commutative by the functoriality of the Leray spectral sequences: here the sign $(-1)$ arise from the difference of degrees. Therefore the desired diagram is commutative.
\end{proof}
\begin{remark}
	In \cite[Thm. 5.4]{shihopurity}, Shiho proved this compatibility for more general regular schemes $X$ over $\Fp$ where the residue fields of $x,y $ are not necessarily to be finite or perfect, but assuming that $n=1$.
\end{remark}

\begin{corollary}[Cohomological purity]
	Assume $X$ and $Z:=X_s$ are as before. For any integer $ i \leq d+1$, there is a canonical isomorphism \[ Gys_{i,n} : H^{d+1-i}(Z, \nu_{n,Z}^d) \xrightarrow{\quad \cong \quad} H^{d+2-i}_Z(X, \Wno_{X,\log}^{d+1})   \]
\end{corollary}
\begin{proof}
	This follows from the spectral sequence
	\[ E^{u,v}_1=H^u(Z, R^vi^!\Wno_{X,\log}^{d+1}) \Rightarrow H^{u+v}_Z(X,\Wno_{X,\log}^{d+1})  \]
	and the above purity theorem.
\end{proof}

\begin{corollary}\label{tracemap} Assume that the residue field $k=\Fp$. Then
	there is a canonical map, called the trace map:
	\[ Tr: H^{d+2}_Z(X, \Wno_{X,\log}^{d+1}) \longrightarrow \Zpn.\]
It is bijective if $Z$ is connected.
\end{corollary}
\begin{proof}
	We define this trace map as $ tr_Z\circ Gys_{i,n}^{-1}$, where $tr_Z$ is the trace map in Theorem \ref{satoduality}.
\end{proof}
We conclude this chapter with the following  compatibility result.
\begin{proposition}
	Let $X, Z$ be as in Theorem \ref{newpurity}, and let $W$ be a smooth closed subscheme of $Z$ of codimension $r$, giving the following commutative diagram:
	 \[ \xymatrix{ W\ar@{^(->}[rr]^{\imath} \ar@{_(->}[dr]_{i_W} && Z  \ar@{^(->}[dl]^i \\
		& X&  }\]
	Then the following diagram:
	\[ \xymatrix{  \Wno_{W, \log}^{d-r} \ar[rr]^{Gys_{W\hookrightarrow Z}}  \ar[dr]_{\theta_{W\hookrightarrow X}}&& R^r\imath^!\nu_{n,Z}^d \ar[dl]^{Gys_{Z \hookrightarrow X}}\\
		& R^{r+1}i_W^!\Wno_{X,\log}^{d+1} &
		}\]
		commutes. Recall that $\theta_{W \hookrightarrow X}$ is the Gysin morphism defined by Shiho in Theorem \ref{shihopurity}, $Gsy_{W \hookrightarrow Z}$ be the Gysin morphism defined by Sato in Theorem \ref{satopurity}, and $Gys_{Z\hookrightarrow X}$ is the Gysin morphism in Theorem \ref{newpurity}.
\end{proposition} 
\begin{proof}
	We may assume $W\subseteq Z_j$, for some irreducible component $Z_j$ of $Z$.  We denote the natural morphisms as in the following diagram: 
	\[ \xymatrix{ W \ar[r]^-{\imath_W}  \ar[dr]_{i_W}&Z_i \ar[r]^{\imath_{Z_j}} \ar[d]^{i_{Z_j}}& Z \ar[dl]^i\\
	& X&  }\]
	Then we have
	\[ \xymatrix@C=.5cm@R=0.3cm{ \Wno_{W}^{d-r} \ar[rrrr]_{(1)} \ar[drr] \ar[dddrr]^{(2)} &&&& R^r\imath^!\nu_{n,Z}^d \ar[dddll]_{(3)}\\
	&& R^r\imath_W^!\Wno_{Z_j}^d \ar[urr] \ar[dd] &&\\
		&&&&\\
		&&R^{r+1}i_W^!\Wno_{X,\log}^{d+1}&&
	}\]
  Sato's Gysin morphisms satisfy a transitivity property, which implies the square (1) commutes. So does (2) due to Shiho's remark in \cite[Rem. 3.13]{shihopurity}. It's enough to show the square (3) commute. Hence we reduced to the case $r=0$. Then we may write $\Wno_{W,\log}^d$ as the kernel of $\oplus_{x\in W^0}i_{x*}\Wno_{x,\log}^d \to \oplus_{y\in W^1}i_{y*}\Wno_{y,\log}^{d-1} $. Then Sato's Gysin morphism will be the identity, and our definition of Gysin morphism locally is exactly that given by Shiho(see the proof of Theorem \ref{newpurity}). 
\end{proof}

\section{Coherent duality}
From now on, we fix the notation as in the following diagram.
$$\xymatrix@!=4pc {
X_s \ar@{^{(}->}[r]^-{i} \ar[d]^{f_s} & X \ar[d]^f &X_{\eta}\ar[d]^{f_{\eta}}\ar@{_{(}->}[l]_-j \\
s \ar@{^{(}->}[r]^-{i_s} &B=\Spec(k [[t]] )&\eta \ar@{_{(}->}[l]_-{j_{\eta}}
}
$$
where $f$ is a projective strictly semistable scheme, $Z=X_s$ is the special fiber, $X_{\eta}$ is the generic fiber, and $k$ is a perfect field of characteristic $p>0$.
\subsection{The absolute differential sheaf $\Omega_X^1$}
We recall some lemmas on local algebras.
\begin{lemma}
	Let $(A,\m, k) \to (A', \m',k')$ be a local morphism of Noetherian local rings. If $A'$ is regular and flat over $A$, then $A$ is regular.
\end{lemma}
\begin{proof}
	Note that flat base change commutes with homology. Hence,   for $q>\text{dim}(A')$, we have  $\text{Tor}_{q}^A(k,k)\otimes_AA'=\text{Tor}_{q}^{A'}(k\otimes A', k\otimes A')=0$.  Since $A'$ is faithful flat over $A$, this implies $\text{Tor}_q^A(k,k)=0$, for $q>\text{dim}(A')$. Therefore, the global dimension of $A$ is finite. Thanks to Serre's theorem \cite[Thm. 19.2]{matsumura}, $A$ is regular. 
\end{proof}
\begin{lemma}
	Let $A$ be a regular local ring of characteristic $p>0$ such that $A^p \to A$ is finite. Then $\Omega_A^1$ is a free $A$-module.
\end{lemma}
\begin{proof}
	By \cite[Thm. 2.1]{kunzregularinp}, regularity implies that $A^p \to A$ is flat, so faithfully flat. The above lemma implies that $A^p$ is regular as well. Then we use a conjecture of Kunz, which was proved in \cite{kimuranittsuma},  that there exists a $p$-basis of $A$. Therefore the assertion follows.
\end{proof}

\begin{proposition}
	The absolute differential sheaf $\Ome_{X}^1$ is a locally free $\mathcal{O}_X$-module of rank $d+1$.
\end{proposition}
\begin{proof}
	Note that we have an exact sequence \[ f^*\Ome_B^1 \to \Ome_X^1 \to \Ome_{X/B}^1 \to 0.\]
	Both $f^*\Ome_B^1$ and $\Ome_{X/B}^1 $ are coherent, so is $\Ome_X^1$. Then we may reduce to local case, and the assertion is clear by the above lemma.
	
\end{proof}

\begin{corollary}
	The sheaves $\Ome_X^i$, $Z\Ome_X^i$ (via F), $\Ome_{X}^i/B\Ome_{X}^i$ (via F) are locally free $\mathcal{O}_X$-modules.
\end{corollary}
\begin{proof}
	As in the smooth case, using Cartier isomorphisms, we can show this inductively.
\end{proof}

\subsection{Grothendieck duality theorem}

The Grothendieck duality theorem studies the right adjoint functor of $Rf_*$ in $D^+_{\text{qc}}(X)$, the derived category of $\mathcal{O}_X$-modules with bounded from below quasi-coherent cohomology sheaves. There are several approaches to this functor.

In our case, we follow a more geometric approach, which was given by Hartshorne in \cite{hartshorneduality}. Here we only use the Grothendieck duality theorem for projective morphisms.
\begin{definition}
	A morphism $g: M \to N$ of schemes is called projectively embeddable if it factors as \[ \xymatrix{M \ar[rr]^g \ar[dr]^{p} &&N \\
		&\mathbb{P}^n_N \ar[ur]^{q}&
		  }\]
		  for some $n \in \mathbb{N}$, where $q$ is the natural projection, $p$ is a finite morphism.
\end{definition}
\begin{remark}\label{projective}
	In our case, the projectivity of the scheme $f: X \to B={\rm Spec}(k[[t]])$ and the fact that the basis is affine imply that $f$ factors through  a closed immersion $X\hookrightarrow \mathbb{P}^n_S$ for some $n\in \mathbb{N}$ [EGA, II 5.5.4 (ii)].
\end{remark}

\begin{theorem}(Grothendieck duality theorem \cite[\S 11]{hartshorneduality})
	Let $g:M \to N$ be a projectively embeddable morphism of noetherian schemes of finite Krull dimension. Then, there exists a functor $f^!: D^+_{qc}(N) \to D^+_{qc}(M)$ such that the following holds:
	\begin{itemize}
	    \item[(i)] If $h: N \to T $ is a second projectively embeddable morphism, then $(g\circ h)^!=h^!\circ g^!$;
		\item[(ii)] If $g$ is smooth of relative dimension $n$, then $g^!(\mathscr{G})=f^*(\mathscr{G})\otimes \Ome_{M/N}^{n}[n]$;
		\item[(iii)] If $g$ is a finite morphism, then \textup{$g^!(\mathscr{G})=\bar{g}^*R\sHom_{\mathcal{O}_N}(g_*\mathcal{O}_M, \mathscr{G} )$}, where $\bar{g}$ is the induced morphism $(M, \mathcal{O}_M) \to (N, g_*\mathcal{O}_M)$;
		\item[(iv)] There is an isomorphism \textup{ \[\theta_g: Rg_*R\sHom_{\mathcal{O}_M}(\mathscr{F}, g^!\mathscr{G}) \xrightarrow{\cong} R\sHom_{\mathcal{O}_N}(Rf_*\mathscr{F}, \mathscr{G}),    \]}
		for  $\mathscr{F} \in D^-_{qc}(M)$, $\mathscr{G} \in D^+_{qc}(N)$.
	\end{itemize}
\end{theorem}

\subsection{Relative coherent duality}


In this subsection, we fix a base scheme $B={\rm Spec}(k[[t]])$, and prove the following relative duality result.

\begin{theorem}\label{fc}
Let $Y$ be a regular scheme and let $f: Y \to B=Spec(k[[t]])$ be a projective morphism of relative dimension $d$. There exists a canonical isomorphism
$$\text{Tr}_f : \Ome_Y^{d+1}[d]  \xrightarrow{\cong}   f^!\Ome_B^1.$$
\end{theorem}
 
 This theorem can be obtained by some explicit calculations.
 
\begin{lemma}\label{regularimmersion}
Let $\iota : Y \hookrightarrow  \mathbb{P}_B^{N}$ be a regular closed immersion with the defining sheaf $\mathcal{I}$. Then the sequence of $\mathcal{O}_Y$-modules 
\[ 0\to  \iota^{*}\mathcal{I} /\mathcal{I}^2 \to \iota^{*}\Ome^1_{\mathbb{P}_B^{N}}  \to \Ome^1_Y \to 0\]
is exact. 
\end{lemma}
\begin{proof}
We only need to show that the left morphism is injective. Let $\mathcal{K}$ be the kernel of the canonical morphism $\iota^{*}\Ome^1_{\mathbb{P}_B^{N}}  \to \Ome^1_Y$. Since both $\iota^{*}\Ome^1_{\mathbb{P}_B^{N}}$ and $ \Ome^1_Y $ are coherent and locally free,  the kernel $\mathcal{K}$ is coherent and flat. Then it follows that $\mathcal{K}$ is also locally free. By counting the ranks, we have rank($\mathcal{I}/\mathcal{I}^2$)=rank($\mathcal{K}$). Thus, the induced surjective morphism $\iota^{*}\mathcal{I}/\mathcal{I}^2 \to \mathcal{K}$ is an isomorphism.
\end{proof}

\begin{proof}[Proof of Theorem \ref{fc}]
By Remark \ref{projective}, we have a decomposition of maps
\[ \xymatrix{ Y \ar@{^{(}->}[r]^-{\iota}  \ar[d]^-f& \mathbb{P}_B^{N} \ar[dl]^-p \\
B & }\]
for some $N\in \mathbb{N}$. Here $\iota$ is a regular closed immersion. Using Koszul resolution of $\mathcal{O}_Y$ with respect to the closed immersion $\iota$, we have
$$\iota^!\Ome_{\mathbb{P}_B^{N}}^{N+1}  \cong \iota^*\Ome_{\mathbb{P}_B^{N}}^{N+1}\otimes \iota^{*}\text{det}(\mathcal{I} /\mathcal{I}^2)^{\vee}[d-N]. $$
 By Lemma  \ref{regularimmersion}, we have 
$$ \iota^*\Ome_{\mathbb{P}_B^{N}}^{N+1}\otimes \iota^{*}\text{det}(\mathcal{I} /\mathcal{I}^2)^{\vee} \cong \Ome_Y^{d+1}.   $$
Since $p$ is smooth, we have isomorphisms 
$$ p^!\Ome_B^1\cong p^!\mathcal{O}_B\otimes p^*\Ome_B^1 \cong \Ome_{\mathbb{P}_B^{N}/B}^N[N] \otimes p^*\Ome_B^1 \cong \Ome_{\mathbb{P}_B^{N}}^{N+1}[N]. $$
Noting that $$ f^!\Ome_B^1=\iota^! p^!\Ome_B^1,$$  the theorem follows.

\end{proof}
\begin{remark}
From the proof, we can see Theorem \ref{fc} is still true in more general situations. In the light of our application, we just proof this simple case and remark that we only use the case  that the residue field $k$ of the base scheme is a finite field $\Fp$.
\end{remark}

\subsection{Grothendieck local duality}

Let $(R,\mathfrak{m})$ be a regular local ring of dimension $n$ with maximal ideal $\mathfrak{m}$, and $R/\mathfrak{m}\cong \Fp$ is a finite field of characteristic $p$.  For any finite $R$-module $M$, we have a canonical pairing
\begin{equation}\label{localpairing}
H_{\m}^i(\Spec(R), M) \times \text{Ext}_{R}^{n-i}(M, \Ome_R^n) \to H_{\m}^n(\Spec(R), \Ome_R^n) \xrightarrow{\text{Res}} \Fp \xrightarrow{tr_{\Fp/\F_p}} \Z/p\Z.
\end{equation}
 
 \begin{theorem}[Grothendieck local duality]
For each $i\geq 0$, the pairing (\ref{localpairing}) induces isomorphisms
$$ \normalfont \text{Ext}^{n-i}_R(M, \Ome_R^n) \otimes_R \hat{R} \cong \text{Hom}_{\Zp1}(H_{\m}^i(\text{Spec}(R), M), \Zp1),$$
$$\normalfont  H_{\m}^i(\Spec(R), M) \cong  \text{Hom}_{\text{cont}}(\text{Ext}^{n-i}_R(M, \Ome_R^n), \Zp1),$$
where $\text{Hom}_{\text{cont}}$ denotes the set of continuous homomorphisms with respect to $\m$-adic topology on Ext group.
\end{theorem}
\begin{proof}
This is slightly different from the original form of Grothendieck local duality in \cite[Thm. 6.3]{hartshornelocal}.  But, in our case, the dualizing module $I=H^n_{\m}(\Spec(R), \Ome_{R}^n)$ can be written as $\varinjlim\limits_{n}\text{Hom}_{\Zp1}(R/\m^n, \Zp1) \subset Hom_{\Zp1}(A,\Zp1)$(cf. Example 3 on Page 67 in loc.cit.). Then we identify
$$Hom_R(H_{\m}^i(\text{Spec}(R), M), I) = Hom_R(H_{\m}^i(\text{Spec}(R), M), \varinjlim\limits_{n}\text{Hom}_{\Zp1}(R/\m^n, \Zp1)) $$
$$=Hom_R(H_{\m}^i(\text{Spec}(R), M), Hom_{\Zp1}(R, \Zp1)) = Hom_{\Zp1}(H_{\m}^i(\text{Spec}(R), M), \Zp1),$$
where the second equality follows from the fact that each element of $H_{\m}^i(\text{Spec}(R), M)$ is annihilated by some power of $\m$. The second isomorphism in the theorem follows from the definition of continuity.

\end{proof}

\subsection{Absolute coherent duality}
In our case, the base scheme $B=\Spec(k[[t]])$ is a complete regular local ring of dimension $1$. We assume $k=\Fp$. Combining Grothendieck local duality on the base scheme $B$ with the relative duality theorem \ref{fc}, we obtain an absolute duality on $X$.
\begin{proposition}
Let $\mathscr{F}$ be a locally free $\mathcal{O}_X$-module on $X$, and let $\mathscr{F}^{t}$ be the sheaf given by \textup{$\Hom(\mathscr{F}, \Ome_{X}^{d+1})$}. Then we have 
\begin{equation}
H_{\m}^i(B , Rf_*\mathscr{F})\cong H^i_{X_s}(X, \mathscr{F});
\end{equation}
\begin{equation}\label{extensionsheaf}
\text{Ext}^{1-i}_{\mathcal{O}_B}(Rf_*\mathscr{F}, \Ome_B^1)  \cong H^{d+1-i}(X, \mathscr{F}^{t}).
\end{equation}
\end{proposition}

\begin{proof}
For the first equation, we have the following canonical identifications:
\begin{equation*}
\begin{array}{rclr}
       H_{\m}^i(B, Rf_*\mathscr{F})& = & H^i(B,i_{s*}Ri_s^!Rf_*\mathscr{F}) &\\
                  & = & H^i(B, i_{s*}Rf_{s*}Ri^!\mathscr{F}) &\\
                  &=& H^{i}(\Fp, Rf_{s*}Ri^!\mathscr{F}) &(i_{s*} \text{is exact})\\
                  &=& H^i(X_s, Ri^!\mathscr{F}) &\\
                  &=& H^i_{X_s}(X, \mathscr{F}) 
      \end{array}
\end{equation*}
where the second equality follows from $Rf_{s*}Ri^! \xrightarrow{\cong} Ri_{s}^!Rf_{*}$ in \cite[Tag 0A9K]{stacks-project}.

For the second,  we have
\begin{equation*}
\begin{array}{rclr}
\text{Ext}^{1-i}_{\mathcal{O}_B}(Rf_*\mathscr{F}, \Ome_B^1) &=& R^{1-i}\Gamma(B, R\kern -.9pt\sHom_{\mathcal{O}_B}(Rf_*\mathscr{F}, \Ome_B^1)) &\\
&=& R^{1-i}\Gamma(B, Rf_*R\kern -.9pt\sHom_{\mathcal{O}_X}(\mathscr{F}, f^!\Ome_B^1)) & \text{(adjunction)}\\
&=& R^{1-i}\Gamma(B, Rf_*R\kern -.9pt\sHom_{\mathcal{O}_X}(\mathscr{F}, \Ome_X^{d+1}[d])) &  \text{(Theorem \ref{fc})}\\
&=& R^{1-i}\Gamma(B, Rf_*\mathscr{F}^{t}[d]) & (\text{definition of}\   \mathscr{F}^{t} )\\
&=& H^{d+1-i}(X, \mathscr{F}^{t})
\end{array}
\end{equation*}
\end{proof}

By taking different $\mathscr{F}$ in the above theorem and using Grothendieck local duality, we obtain the following corollaries.

\begin{corollary}\label{omegaduality}
The natural pairing 
$$H^i(X, \Ome^j_{X})\times H^{d+1-i}_{X_s}(X, \Ome^{d+1-j}_{X}) \rightarrow H^{d+1}_{X_s}(X, \Ome^{d+1}_{X})\xrightarrow{tr} \Z/p\Z$$
induces isomorphisms
$$ H^{i}(X, \Ome_X^{j}) \cong \text{Hom}_{\Zp1}(H^{d+1-i}_{X_s}(X, \Ome_X^{d+1-j}) ,\Zp1),$$
$$ H^{d+1-i}_{X_s}(X, \Ome_X^{d+1-j}) \cong \text{Hom}_{\text{cont}}(H^{i}(X, \Ome_X^{j}) ,\Zp1).$$
i.e., it is a perfect pairing of topological $\Zp1$-modules if we endow $H^{i}(X, \Ome_X^{j})$ with the  $\m$-adic topology and $H^{d+1-i}_{X_s}(X, \Ome_X^{d+1-j}) $ with the discrete topology.
\end{corollary}
\begin{remark}
	For the first isomorphism, we may ignore the topological structure on cohomology groups, and  view it as an  isomorphism of $\Zp1$-modules. In the next chapter, we only use this type of isomorphisms(cf. Proposition \ref{onesideiso}).
\end{remark}

\begin{corollary} \label{zomegaduality}
The natural pairing 
$$H^i(X, Z\Ome^j_{X})\times H^{d+1-i}_{X_s}(X, \Ome_X^{d+1-j}/B\Ome_{X}^{d+1-j}) \rightarrow H^{d+1}_{X_s}(X, \Ome^{d+1}_{X})\xrightarrow{tr} \Z/p\Z$$
is a perfect pairing of topological $\Zp1$-modules, if we endow the cohomology groups with the topological structures as in the above corollary.
\end{corollary}
 \begin{proof}
 The fact that $Z\Ome^j_{X}$ is dual to $\Ome_X^{d+1-j}/B\Ome_{X}^{d+1-j}$ is similar as \cite[Lem. 1.7]{milnesurface}.
 \end{proof}
Furthermore, we can do the same thing for twisted logarithmic Kähler differential sheaves.

Let $X$ be as before , and let  $\jmath : U \hookrightarrow X$ be the complement of a reduced divisor $D$ on $X$ with simple normal crossings.
Let $D_1, \cdots, D_s$ be the irreducible components of $D$. For $\underline{m}=(m_1,\cdots, m_s) $ with $m_i \in \mathbb{Z}$ let 
\begin{equation}\label{defnofD1} 
mD={\underline{m}}D=\sum_{i=1}^s m_i D_i \end{equation}
be the associated divisor.   

\begin{definition}
For the above defined $D$ on $X$ and $j \geq 0$, $m=\underline{m} \in \mathbb{Z}^s$ , we set 
\begin{equation}\label{twistedcoh}
\Ome_{X|mD}^j=\Ome_X^{j}(\log D)(-mD)=\Ome_X^{j}(\log D) \otimes \mathcal{O}_X(-mD)
\end{equation}
where $\Ome_X^{j}(\log D)$ denotes the sheaf of absolute Kähler differential $j$-forms on $X$ with logarithmic poles along $|D|$.  Similarly, we can define $Z\Ome_{X|mD}^j$, $B\Ome_{X|mD}^j$.
\end{definition}
\begin{remark}
Note that $\Ome_{X|D}^{d+1}=\Ome_X^{d+1}\otimes \mathcal{O}_X(D)\otimes \mathcal{O}_X(-D)=\Ome_X^{d+1}$, where $d$ is the relative dimension of $X$ over $B$.
\end{remark}

\begin{corollary}\label{twistedomega}
The natural pairing 
$$H^i(X, \Ome^j_{X|-mD})\times H^{d+1-i}_{X_s}(X, \Ome_{X|(m+1)D}^{d+1-j}) \rightarrow H^{d+1}_{X_s}(X, \Ome^{d+1}_{X})\xrightarrow{tr} \Z/p\Z$$
is a perfect pairing of topological $\Zp1$-modules, if we endow the cohomology groups with the topological structures as before.
\end{corollary}
\begin{proof} \label{logcoherentpairing}
Note that the pairing 
\begin{equation}
\Ome_X^{j}(\log D)(mD) \otimes   \Ome_X^{d+1-j}(\log D)((-m-1)D) \to \Ome_X^{d+1}
\end{equation}
 is  perfect.
\end{proof}
Similarly, we define

\begin{equation}
Z\Ome^j_{X|-mD}=\ker(d\colon \Ome^j_{X|-mD} \to \jmath_*\Ome^{j+1}_U );
\end{equation}
\begin{equation}
B\Ome^j_{X|mD}=\mathrm{Image}(d\colon \Ome^{j-1}_{X|mD} \to \Ome^{j}_X ).
\end{equation}

 We have the following result.
\begin{corollary} \label{twistedzomega}
The natural pairing 
$$H^i(X, Z\Ome^j_{X|-mD})\times H^{d+1-i}_{X_s}(X, \Ome_{X|(m+1)D}^{d+1-j}/B\Ome_{X|(m+1)D}^{d+1-j}) \rightarrow H^{d+1}_{X_s}(X, \Ome^{d+1}_{X})\xrightarrow{tr} \Z/p\Z$$
is a perfect pairing of topological $\Zp1$-modules, if we endow the cohomology groups with the topological structures as before.
\end{corollary}

\section{Duality}
In the rest of this paper, we assume the residue field $k$ of the base scheme $B$ is a finite field $\Fp$.
Recall that $f: X \to B=\text{Spec}(\Fp[[t]])$ is a projective strictly semistable scheme of relative dimension $d$.
In this section, we will prove two duality theorems. The first one is for $H^i(X, \Wno^j_{X,\log})$, which we call unramified duality. The second is for $H^i(U, \Wno^j_{U,\log})$, where $U$ is the open complement of a reduced effective Cartier divisor with Supp$(D)$ has simple normal crossing. We call it the ramified duality. 
\subsection{Unramified duality}
The product on logarithmic de Rham-Witt sheaves
$$ \Wno_{X,\log}^j \otimes \Wno_{X,\log}^{d+1-j} \to \Wno_{X,\log}^{d+1}$$
induces a pairing
$$ i^{*}\Wno_{X,\log}^j \otimes^{\mathbb{L}}  Ri^{!}\Wno_{X,\log}^{d+1-j} \to Ri^{!}(\Wno_{X,\log}^j \otimes \Wno_{X,\log}^{d+1-j})  \to Ri^{!}\Wno_{X,\log}^{d+1},   $$
where the first morphism is given by the  adjoint  map of the diagonal map $\phi$ in the following diagram \[ \xymatrix{ Ri_*( i^{*}\Wno_{X,\log}^j \otimes^{\mathbb{L}}  Ri^{!}\Wno_{X,\log}^{d+1-j}) \ar[r]^{\cong} \ar[dr]_{\phi} &\Wno_{X,\log}^j \otimes^{\mathbb{L}}  Ri_*Ri^{!}\Wno_{X,\log}^{d+1-j} \ar[d]^{\text{adj}} \\
	& \Wno_{X,\log}^j \otimes \Wno_{X,\log}^{d+1-j}. }\] Here the isomorphism is given by the projection formula.

Apply $R\Gamma(X_s, \cdot)$ and the proper base change theorem(SGA4$\frac{1}{2}$, \cite[Arcata IV]{SGA41/2}), we have a pairing 
\begin{equation} \label{pairingwithoutfiltration}
H^i(X, \Wno^j_{X,\log})\times H^{d+2-i}_{X_s}(X, \Wno^{d+1-j}_{X,\log}) \rightarrow H^{d+2}_{X_s}(X, \Wno^{d+1}_{X,\log})\xrightarrow{Tr} \Zpn,
\end{equation} 
where the trace map Tr is given by Corollary \ref{tracemap}.
\begin{theorem}\label{unramifiedduality}
The pairing (\ref{pairingwithoutfiltration}) induces an isomorphism
\[ H^i(X, W_n\Ome_{X,\log}^j)  \xrightarrow{\cong} Hom_{\Zpn}(H^{d+2-i}_{X_s}(X, W_n\Ome_{X,\log}^{d+1-j}),\Zpn).\]
of $\Zpn$-modules. If we endow $H^{d+2-i}_{X_s}(X, W_n\Ome_{X,\log}^{d+1-j})$ with the discrete topology, and endow $H^i(X, W_n\Ome_{X,\log}^j)$ with the compact-open topology, we get an isomorphism, by the Pontryagin duality
\[  H^{d+2-i}_{X_s}(X, W_n\Ome_{X,\log}^{d+1-j}) \xrightarrow{\cong} Hom_{\Zpn,\mathrm{cont}}(H^i(X, W_n\Ome_{X,\log}^j),\Zpn).\]
\end{theorem}
\begin{proof}\renewcommand{\qedsymbol}{}
By the exact sequence (i) in Proposition \ref{logtocoh}, the problem is reduced to the case $n=1$. In this case, we use the classical method as in \cite{milneduality}, i.e., using the exact sequence (ii) and (iii) in Proposition \ref{logtocoh}, we reduce the problem to coherent duality. Before we do this, we have to check the compatibility between the trace maps.  It is enough to do this on the base scheme $B$, by the definitions of trace map and residue map(cf. (\ref{localpairing})).
\end{proof}
\begin{proposition}
The following diagram
\[ \xymatrix@C=1.5cm@R=1.5cm{ H^2_{\m}(B, \Ome_{B,\log}^1) \ar[r]^-{\text{Tr}} & \Zp1 \ar@{=}[d]\\
  H^1_{\m}(B, \Ome_{B}^1) \ar[u]^{\delta} \ar[r]^-{\text{Res}} & \Zp1 }   \]
  commutes, where $\delta$ is the connection map induced by the following exact sequence
  \[  0\to \Ome_{B, \log}^1 \to \Ome_B^1 \to \Ome_B^1 \to 0.\]
  \end{proposition}

\begin{proof}
We have the following diagram
\[ \xymatrix@C=0.3cm@R=0.3cm{ H^2_{\m}(B, \Ome_{B,\log}^1)  \ar[rrrrrr]^{\text{Tr}} &&&&&& \Zp1 \ar@{=}[ddddddd]\ar@{=}[ddll]\\
&&&(1)&&&\\
&& H^1(\Fp, \Ome_{\Fp, \log}^0) \ar[uull]^{Gys_{i_s}}_{\cong}  \ar[rr]^{\cong}& &\Zp1 \ar@{=}[dd]&&\\
&(2) &&(3)&&& \\
&& H^0(\Fp, \Ome_{\Fp}^0) \ar[dddll]_{\text{Gys}_{i_x}}\ar[uu]^{\delta} \ar[rr]^{tr_{\Fp/\F_p}} && \Zp1 \ar[dl]^{\varphi}\ar@{=}[dddrr]&&\\
&& (4) &\Ome_K^1/\Ome_B^1 \ar[ddlll]^{\cong}\ar[ddrrr]_{\text{ Tate Residue}} & (5) && \\
&&&(6)&&&\\
H^1_{\m}(B, \Ome_{B}^1)\ar[uuuuuuu]^{\delta} \ar[rrrrrr]^{\text{Res}} &&&&& &\Zp1
} \]
where the morphism $\varphi: a\mapsto ad\log(t)$, and $K=\Fp((t))$. 

The diagrams (1) and (6) are commutative by the definition of Tr and Res. (2) is commutative by the functoriality of $\delta$. That classical Milne duality is compatible with coherent duality implies (3) is commutative. The local description of the Gysin map will imply that the diagram (4) commutes \cite[Lem. 3.5]{shihopurity} , \cite[II (3.4)]{groschernclass}. The diagram (5) commutes by the explicit definition of $\varphi$ and  the definition of the Tate residue map. 
\end{proof}
\begin{proof}[Proof of Theorem \ref{unramifiedduality}(cont.)]
By taking the cohomolgy groups of the exact sequences (ii) and (iii) in Proposition \ref{logtocoh}, we have the following commutative diagram with exact rows:
\[  \xymatrix@C=.3cm@R=.5cm{ \cdots \ar[r]  &  H^i(X, \Ome_{X,\log}^j) \ar[r] \ar[d] & H^i(X, Z\Ome_{X}^j) \ar[r]\ar[d]^{\cong} & H^i(X, \Ome_X^j) \ar[r]\ar[d]^{\cong} & \\
\cdots \ar[r]  & H^{d+2-i}_{X_s}(X, \Ome_{X,\log}^{d+1-j})^* \ar[r] & H^{d+1-i}_{X_s}(X, \Ome_X^{d+1-j}/B\Ome_X^{d+1-j})^* \ar[r] &H^{d+1-i}_{X_s}(X, \Ome_X^{d+1-j})^*\ar[r] &
}\]
where $M^*$ means $\text{Hom}_{\Zp1}(M, \Zp1)$, for any $\Zp1$-module $M$.

The isomorphisms for cohomology groups of $Z\Ome_{X}^j$ and $\Ome_{X}^j$ are from the coherent duality theorem, see Corollary \ref{zomegaduality} and \ref{omegaduality}. Hence, we have \[  H^i(X, \Ome_{X,\log}^j)  \xrightarrow{\cong} H^{d+2-i}_{X_s}(X, \Ome_{X,\log}^{d+1-j})^*.  \]
\end{proof}
\begin{remark}\label{thank_referee}
\begin{itemize}
	\item[(i)] Note that $H^i(X,W_n\Omega^r_{X,\log})$ is not finite in general. In Theorem \ref{unramifiedduality}, if  $H^{d+2-i}_{X_s}(X, W_n\Ome_{X,\log}^{d+1-j})$ is finite (e.g., $j=0$), then we get a perfect pairing of finite $\Z/p^n\Z$-modules.
	\item[(ii)] For the case of $j=0$, all the cohomology groups in Theorem \ref{unramifiedduality} are finite, by using the purity theorem \ref{newpurity}, the above pairing agrees with that in \cite[Thm. 1.2.2 ]{satoncv}.
\end{itemize}
\end{remark}

\subsection{Relative Milnor $K$-sheaf}
On a smooth variety over a field, the logarithmic de Rham-Witt sheaves are closely related to the Milnor $K$-sheaves via the Bloch-Gabber-Kato theorem \cite{blochkato}. In this section, we first recall some results on the Milnor $K$-sheaf on a regular scheme, and then define the relative Milnor $K$-sheaf with respect to some divisor $D$ as in \cite[\S2.3]{ruellingsaito}. At last, we show the Bloch-Gabber-Kato theorem still holds on a regular scheme over $\F_p$. This result is well known to the experts but due to the lack of reference, we give a detailed proof.

In this section, we fix $Y$ to be a connected regular scheme over $\F_p$ of dimension $d$ (cf. Remark \ref{connectedness} below). 

\begin{definition}
	For any integer $r$, we define the $r$th-Milnor K-sheaf $\mathcal{K}^M_{r,Y_{\text{Zar}}}$ to be the sheaf 
	\[ T \mapsto \text{Ker}\left(i_{\eta*}K^M_r(\kappa(\eta)) \xrightarrow{\partial} \bigoplus_{x \in Y^{1}\cap T}i_{x*}K^M_{r-1}(\kappa(x)\right)\]
	on the Zariski site of $Y$, where $T$ is any open subset of $Y$,  $\eta$ is the generic point of $Y$, $i_x: x\to T$ is the natural map and $\partial$ is the sheafified residue map of Milnor K-theory of fields (cf. \S 1.3.2). The sheaf $\mathcal{K}^M_{r,Y}$ is the associated sheaf of $\mathcal{K}^M_{r,Y_{\text{Zar}}}$ on the (small) \'etale site of $Y$.
\end{definition}
In particular, $\mathcal{K}^M_{r,Y}=0$ for $r< 0$, $\mathcal{K}^M_{0,Y}=\Z$ and $\mathcal{K}^M_{1,Y}=\mathcal{O}_X^{\times}$.
 \begin{remark}\label{connectedness} The connectedness assumption on $Y$ is just for simplification our notations.
 In case $Y$ is not connected, we write $Y=\coprod_jY_j$ as disjoint union of its connected components. Then define  \[ \mathcal{K}^M_{r,Y}:=\bigoplus_ji_{Y_j*}\mathcal{K}^M_{r,Y_j}  \]
 where $i_{Y_j}: Y_j \to Y$ be the natural map. The results in this section still hold for non-connected $Y$.
 \end{remark}
We define $\mathcal{K}^{M,\text{naive}}_{*, Y}$  to be the sheaf on the  \'etale site of Y associated to the functor 
\[  A \mapsto \bigoplus_{i\geq 0}(A^{\times})^{\otimes i}/<a\otimes b |a+b=1>\]
from commutative rings to graded rings.
\begin{theorem}(\cite[Thm. 1.3]{kerzgersten}, \cite[Thm. 13, Prop. 10]{kerzmilnorfinite})\label{kerzthm} Let $Y$ be a connected regular scheme over $\F_p$.
\begin{itemize}
\item[(i)] The natural homomorphism  $ \mathcal{K}^{M,\text{naive}}_{r,Y}  \to  \mathcal{K}^M_{r,Y} $  is surjective;
\item[(ii)] If the residue fields at all point of $Y$ are infinite, the map $\mathcal{K}^{M,\text{naive}}_{r, Y}  \to \mathcal{K}^M_{r,Y}$ is an isomorphism;
\item[(iii)] The Gersten complex
\[ 0\to \mathcal{K}^{M}_{r,Y} \to i_{\eta*}K^M_{r}(\kappa(\eta)) \to \bigoplus_{x \in Y^{1}}i_{x*}K^M_{r-1}(\kappa(x)) \to  \bigoplus_{x \in Y^{2}}i_{x*}K^M_{r-2}(\kappa(x))  \to \cdots\]
is universally exact(see Definition \ref{universalexact} below) on the \'etale site of $Y$.
\end{itemize}
\end{theorem}
\begin{definition}\label{universalexact}
Let $A' \to A \to A''$ be a sequence of abelian groups. We say this sequence is universally exact if $F(A') \to F(A) \to F(A'')$ is exact for every additive functor $F: Ab \to \mathcal{B}$ which commutes with filtered small limit. Here $Ab$ is the category of abelian groups and $\mathcal{B}$ is an abelian category satisfying AB5 (see \cite{grothendiecktohoku}).
\end{definition}
\begin{remark}
\item[(i)] Another way to define the Milnor $K$-sheaf \cite{kerzmilnorfinite} is to first define the naive $K$-sheaf $\mathcal{K}^{M,\text{naive}}_{r, Y}$ as above, and then the (improved) Milnor $K$-sheaf as the universal continuous functor associated to the naive $K$-sheaf. By Theorem \ref{kerzthm} (iii), this definition agree with ours.
\item[(ii)] Theorem \ref{kerzthm} (i) implies that $\mathcal{K}^M_{r,Y}$ is \'etale locally generated by the symbols of the form $\{x_1, \cdots, x_r \}$ with all $x_i \in \mathcal{O}_Y^{\times}$.
\item[(iii)] The functor $-\otimes_{\Z} \Z/p^n\Z$ is an additive functor and it commutes with filtering small limit, so the universal exactness property of Gersten complex implies the following sequence
\[  0\to \mathcal{K}^{M}_{r,Y}/p^n \to i_{\eta*}K^M_{r}(\kappa(\eta))/p^n \to \bigoplus_{x \in Y^{(1)}}i_{x*}K^M_{r-1}(\kappa(x))/p^n \to \cdots \]
is exact.
\end{remark}

Let $D=\cup_{i=1}^sD_i$ be a reduced effective Cartier divisor on $Y$ such that Supp$(D)$ has simple normal crossing, let $D_1,\cdots, D_s$ be the irreducible components of $D$, and let $\jmath: U:=Y-D \hookrightarrow Y$ be the open complement.  For $\underline{m}=(m_1,\cdots, m_s) $ with $m_i \in \mathbb{N}$ let 
\begin{equation*}
mD={\underline{m}}D=\sum_{i=1}^s m_i D_i 
\end{equation*}
be the associated divisor.  On  $\mathbb{N}^s$, we define a semi-order as follows:
\[  {\underline{m'}} \geq    \underline{m'} \quad \text{if} \quad m'_i \geq m_i \quad \text{for all $i$.}\]
Using this semi-order, we denote
\begin{equation*} 
{\underline{m'}}D \geq {\underline{m}} D \quad  \text{if} \quad \underline{m'} \geq \underline{m}.   
\end{equation*} 
By the above theorem, we may define the relative Milnor $K$-sheaves with respect to $mD$ using symbols(cf. \cite[Def. 2.7]{ruellingsaito}):
\begin{definition}
For $r\in \Z, m \in \mathbb{N}^s$, we define the Zariski sheaf $\mathcal{K}^M_{r,Y|mD,\text{Zar}}$ to be the image of the following map
\begin{align*}
\text{Ker}(\mathcal{O}_Y^{\times}\to \mathcal{O}_{mD}^{\times}) \otimes_{\Z} \jmath_*\mathcal{K}^M_{r-1,U_{\text{Zar}}} &\to \jmath_*\mathcal{K}^M_{r,U_{\text{Zar}}}\\
x \otimes \{ x_1, \cdots, x_{r-1}\} &\mapsto \{x,x_1,\cdots,x_{r-1} \}
\end{align*}
and define $\mathcal{K}^M_{r,Y|mD}$ to be its associated sheaf on the \'etale site. 
\end{definition}
By this definition, it is clear that $\jmath_*\mathcal{K}^M_{r,Y|mD}=\mathcal{K}^M_{r,U}$, for any $m\in \mathbb{N}^s$.
\begin{proposition}(\cite[Cor. 2.13]{ruellingsaito})\label{inclusionK}
If $m' \geq m$, then we have the inclusions of \'etale sheaves 
\[  \mathcal{K}^M_{r,X|m'D} \subseteq \mathcal{K}^M_{r,X|mD} \subseteq \mathcal{K}^M_{r,X}. \]
\end{proposition}
\begin{proof}
The statement in \cite{ruellingsaito} is on the Zariski and Nisnevich sites, and in particular,  works on the \'etale site.
\end{proof}

Now we recall the Bloch-Gabber-Kato theorem on regular schemes.
\begin{lemma}\label{dlogmap}
There is a natural map $d\log \left[-\right]_n: \mathcal{K}^M_{r,Y} \to \Wno_Y^r$.
\end{lemma}
\begin{proof}\renewcommand{\qedsymbol}{}
If all the residue fields of $Y$ are infinite, this is clear by Theorem \ref{kerzthm}. If not, we can still construct this map via the following local computations. We first define this for any ring $A$.
 Recall that the (improved) Milnor $K$-theory of $A$ can be defined by the first row of the following diagram \cite{kerzmilnorfinite}:
\[  \xymatrix{ 0 \ar[r] &\hat{K}^M_r(A) \ar[r] & K^M_r(A(t)) \ar[r]^-{i_{1*}-i_{2*}} \ar[d]^{d\log[-]_n}& K^M_r(A(t_1,t_2)) \ar[d]^{d\log[-]_n}\\
0 \ar[r] & M_n \ar[r] & W_n\Omega_{A(t)}^r  \ar[r]^-{i_{1*}-i_{2*}} & W_n\Omega_{A(t_1, t_2)}^r
} \]
Here $A(t)$ is the rational function ring over $A$, that is $A[t]_S$ the localization of the one variable polynomial ring with respect to the multiplicative set $S=\{ \sum_{i\in I} a_it^i | \left< a_i \right>_{i\in I}=A \}$.  The assertion follows from the following claim.
\end{proof}
\textbf{Claim}  $M_n=W_n\Omega_A^r$.

Once we have the $d\log$ map for any ring, we will get a map on the Zariski site by sheafification. The desired map is obtained by taking the associated map on the \'etale site. Now it suffices to prove the above Claim.
\begin{proof}[Proof of Claim]
We first assume $n=1$, and in this case we have 
\[   A(t) \otimes_{A} \Omega_A^1 \oplus A(t) \xrightarrow{\cong} \Omega_{A(t)}^1 \]
\[ (a\otimes w, b) \mapsto a\cdot w+b\cdot dt\]
 Using this explicit expression, we can show the kernel of $i_{1*}-i_{2*}$ is $\Omega_A^r$. For general $n$, we first noted that $W_n\Omega_A^r \subseteq M_n$, and we can prove the claim by induction on $n$ via the exact sequence
 \[  0\to V^n\Ome_A^r+dV^n\Ome_A^{r-1} \to W_{n+1}\Ome_A^r \xrightarrow{R} \Wno_A^r\to 0 \]
 as $i_{1*}$ and $i_{2*}$ commute with $R$ and $V$.
\end{proof}
By Definition \ref{definitionOfLog}, the image of $d\log[-]_n$ contains in $\Wno_{Y,\log}^r$. It is clear that the map $d\log[-]_n$ factors through $\mathcal{K}^M_{r,Y}/p^n$. Therefore we have the following result. 
\begin{proposition}\label{blochgabberkato}
The natural map in Lemma \ref{dlogmap} induces an isomorphism 
\[ d\log \left[-\right]: \mathcal{K}^M_{r,Y}/p^n \xrightarrow{\cong} \Wno_{Y,\log}^r.\]
\end{proposition}
\begin{proof}
It is enough to show this map is injective. This is a local question, so we may assume $Y=\text{Spec}(A)$ is a regular local ring over $\F_p$. The N\'eron-Popescu desingularization theorem below tells us that we can assume $Y=\text{Spec}(A)$ is smooth over $\F_p$. Then we have the following commutative diagram 
\[ \xymatrix{   \mathcal{K}^M_{r,Y}/p^n \ar@{^(->}[r] \ar[d]_{d\log[-]_n} & i_{\eta*}K^M_r(\kappa(\eta))/p^n \ar[d]^{\cong}\\
	\Wno_{X,\log}^r \ar@{^(->}[r] & i_{\eta*}\Wno_{\eta,\log}^r }\]
of \'etale sheaves on $Y$. The injection on the first row follows from Theorem \ref{kerzthm}(iii), and Remark \ref{smoothlog} implies the injection on the second row. Therefore the assertion follows from the fact that the right vertical map is an isomorphism, which is  given by the classical Bloch-Gabber-Kato theorem\cite{blochkato}.
\end{proof}

\begin{theorem}[N\'eron-Popescu, \cite{swanneron}]\label{popescu}
Any regular local ring $A$ of characteristic $p$ can be written as a filtering colimit $\varinjlim\limits_{i}A_i$ with each $A_i$ is smooth (of finite type) over $\Fp$.
\end{theorem}

\subsection{Relative logarithmic de Rham-Witt sheaves}
Assume $X \to \text{Spec}(\Fp[[t]])$ is as before, i.e., a projective strictly semistable scheme over $B=\text{Spec}(\Fp[[t]])$. Let $D=\cup_{i=1}^sD_i$ be a reduced effective Cartier divisor on $Y$ such that Supp$(D)$ has simple normal crossing,  let $D_1,\cdots, D_s$ be the irreducible components of $D$, and let $\jmath: U:=Y-D \hookrightarrow Y$ be the open complement of $D$. For $\underline{m}=(m_1,\cdots, m_s) $ with $m_i \in \mathbb{N}$ let 
\begin{equation}\label{defnofD} 
mD={\underline{m}}D=\sum_{i=1}^s m_i D_i 
\end{equation}
be the associated divisor.  In the previous section,  we defined a semi-order on $\mathbb{N}^s$ and 
denoted
\begin{equation}\label{order}  
{\underline{m'}}D \geq {\underline{m}} D \quad  \text{if} \quad \underline{m'} \geq \underline{m}.   
\end{equation} 

\begin{definition}\label{twistedlog}
For $r\geq 0, n\geq 1$, we define \[ \Wno_{X|mD,\log}^r \subset \jmath_*\Wno_{U,\log}^r  \] to be the \'etale additive subsheaf generated \'etale locally by sections
\[ d\log [x_1]_n \wedge \cdots \wedge d\log [x_r]_n \quad \text{with } \quad x_i \in \jmath_*\mathcal{O}_U^{\times}, \ \text{for all $i$, and}\    x_1 \in \ker(\mathcal{O}_X \to \mathcal{O}_D), \]
where $[x]_n=(x,0,\cdots,0) \in W_n(\mathcal{O}_X)$ is the Teichm\"uller representative of $x\in \mathcal{O}_X$, and the symbol $d\log[x]_n =\frac{d[x]_n}{[x]_n}$ as before. 
\end{definition}

\begin{corollary}
For any $m \in \mathbb{N}^s$, $\jmath^*\Wno_{X|mD,\log}^r=\Wno_{U,\log}^r$.	
If $m' \geq m$, then we have the inclusions of \'etale sheaves 
\[  \Wno_{X|m'D,\log}^r \subseteq \Wno_{X|mD,\log}^r \subseteq \Wno_{X,\log}^r. \]

\end{corollary}
\begin{proof}
By definition, $\Wno_{X|mD,\log}^r $ is the image of $\mathcal{K}^M_{r,X|mD}$ under the $d\log[-]_n$. Hence the first claim is clear and the second follows from Proposition \ref{inclusionK} and \ref{blochgabberkato}.
\end{proof}

\begin{theorem}(\cite[Thm. 1.1.6]{jszduality})\label{twistedexact}
There is an exact sequence of \'etale sheaves on X: 
\[ 0 \to W_{n-1}\Ome_{X|[m/p]D,\log}^{r} \xrightarrow{p} \Wno_{X|mD,\log}^r \xrightarrow{R} \Ome_{X|mD,\log}^r \to 0 ,\]
where $[m/p]D=\sum_{i=1}^s[m_i/p]D_i$, and $[m_i/p]=\text{min}\{ m\in \mathbb{N}| m\geq m_i/p \}. $
\end{theorem}
\begin{proof}
This is a local problem, and the local proof in \cite{jszduality} also works in our situation. The idea is reduced to a similar result for Milnor $K$-groups by the Bloch-Gabber-Kato theorem. Then the graded pieces (with respect to $m$) on Milnor $K$-groups can be represented as differential forms, which will give the desired exactness.
\end{proof}
If $m'D\geq mD$ , then  the relation $\mathcal{O}_X(-m'D) \subseteq \mathcal{O}_X(-mD)$ induces a natural transitive map $\Wno_{X|m'D,\log}^r \hookrightarrow \Wno_{X|mD,\log}^r$. This gives us a pro-system of abelian \'etale sheaves $ \lq\lq \varprojlim\limits_m" W_{n}\Ome_{X|mD, \log}^r$.
\begin{corollary}\label{exactofprosystem}
The following sequence is exact
\begin{equation*}
0 \to \lq\lq \varprojlim_m" W_{n-1}\Ome_{X|mD, \log}^r \to  \lq\lq \varprojlim_m" W_{n}\Ome_{X|mD, \log}^r  \to  \lq\lq \varprojlim_m" \Ome_{X|mD, \log}^r  \to 0
\end{equation*}
as pro-objects, where $\lq\lq \varprojlim\limits_m" $ is the pro-system of sheaves defined by the ordering between the $D$'s, which is defined in (\ref{order}). 
\end{corollary}

In \cite{jszduality}, using the filtered de Rham-Witt complexes, we define a pairing between $\Wno_{U, \log}^r$ and the pro-system $\lq\lq \varprojlim\limits_m" \Wno_{X|mD, \log}^{d+1-r}$. This pairing induces a pairing on the cohomology groups.  In this paper, we give an alternative way to define the pairing between $H^i(U,\Wno_{U, \log}^r)$ and $ \varprojlim\limits_{m} H^{d+2-i}_{X_s}(X, \Wno_{X|mD,\log}^{d+1-r}) $, which can be done without introducing the filtered de Rham-Witt complexes \cite[\S 2]{jszduality}. 

\begin{theorem}\label{pairingcoherentlog} The wedge product on de Rham-Witt complexes induces natural maps
\begin{equation}\label{pairing.Hom.1}
\jmath_*\Wno_{U}^r \to \sHom(\lq \lq \varprojlim_m" \Wno_{X|mD, \log}^{d+1-r}, \Wno_{X}^{d+1})
=\varinjlim\limits_m \sHom(\Wno_{X|mD, \log}^{d+1-r}, \Wno_{X}^{d+1});
\end{equation}
\begin{equation}\label{pairing.Hom.2}
\jmath_*Z_1\Wno_{U}^r  \to \varinjlim_m \sHom(\Wno_{X|mD, \log}^{d+1-r}, W_n\Omega_X^d).
\end{equation}
\end{theorem}

For the proof, we need some calculations with Witt vectors.

\begin{lemma}\label{lem;giesserhe} (\cite[Lem. 1.2.3]{geisserhe})
Let $R$ be any ring, and $t\in R$, then $ [1+t]_{m}-[1]_m=(y_0,\cdots, y_{m-1})$ with $y_i\equiv t$ mod $ t^2R$ for $0 \leq i \leq m-1.$ Here $[x]_m=(x,0,\cdots,0) \in W_m(R)$ is the Teichm\"uller representative of $x\in R$.
\end{lemma}
As a consequence, we have
\begin{corollary} 
Let $A$ be an $\Fp$-algebra, and let $a, t\in A$. Then \[ [1+ta]_{m}-[1]_m=(y_0,\cdots, y_{m-1}) \]
with  $y_{i} \in tA$ for  $0\leq i \leq m-1$. 
\end{corollary}

\begin{corollary}
With the notations as above, we have $d[1+tA]_m = d(y_0,\ldots,y_m)$ with $y_i \in tA$.
\end{corollary}
\begin{proof}
We have $dx = 0$ for $x\in W_m(\mathbb F_p)$.
\end{proof}

\begin{corollary}
With the notations above we have the following formula, for $t\in A$
\begin{equation*}
       d[1+t^{p^{(m-1)}n}]_{m} =  d(t^{p^{(m-1)}n}y_0,\cdots, t^{p^{(m-1)}n}y_{m-1})=d([t]_m^{n}\cdot (c_0,\cdots, c_{m-1})).
\end{equation*}
\end{corollary}

\begin{proof}
	The last equality is due to the fact that $[t]_m \cdot (y_0,\ldots,y_{m-1})=(ty_0,t^py_1,\ldots, t^{p^{m-1}}y_{m-1})$.
\end{proof}

\begin{proof}[Proof of Theorem \ref{pairingcoherentlog}]
Let $\alpha$  be a given local section of $\Wno_{U}^r$, we need to find a suitable $m$ such that, for any local section $\beta$ of $\Wno_{X|mD,\log}^{d+1-r}$, $\alpha \wedge \beta $ is a local section of $\Wno_{X}^{d+1}$. This is equivalent to show that we can find $m$ such that, for any $a_{1} \in 1+\mathcal{O}_X(-mD) $  the coefficient (with respect to a local basis) of \[ \alpha  \wedge \frac{d[1+a_{1}]_n}{[1+a_{1}]_n} \wedge \frac{d[a_{2}]_n}{[a_{2}]_n} \wedge \cdots \wedge \frac{d[a_{d+1-r}]_n}{[a_{d+1-r}]_n}  \]
lies in $W_n(\mathcal{O}_X)$. By the above Corollary, this is possible if we take $m$ big enough to eliminate the \lq \lq poles" of $\alpha$  along $D$. This gives us the map \ref{pairing.Hom.1}, and the map \ref{pairing.Hom.2} is defined similarly.
\end{proof}

In  \cite{milnesurface},  Milne defined a pairing of two-term complexes as follows:

Let $$\mathscr{F}^{\bullet}=(\mathscr{F}^0 \xrightarrow{d_{\mathscr{F}}} \mathscr{F}^1),  \quad \mathscr{G}^{\bullet}=(\mathscr{G}^0 \xrightarrow{d_{\mathscr{G}}} \mathscr{G}^1)$$
and $$ \mathscr{H}^{\bullet}=(\mathscr{H}^0 \xrightarrow{d_{\mathscr{H}}} \mathscr{H}^1) $$
be two-term complexes.  A pairing of two-term complexes
\[\mathscr{F}^{\bullet} \times \mathscr{G}^{\bullet}  \to \mathscr{H}^{\bullet}\]
is a system of pairings  \[ <\cdot, \cdot>_{0,0}^0 :  \mathscr{F}^0 \times \mathscr{G}^0 \to \mathscr{H}^0;    \]
\[ <\cdot, \cdot>_{0,1}^1 :  \mathscr{F}^0 \times \mathscr{G}^1 \to \mathscr{H}^1;    \]
\[ <\cdot, \cdot>_{1,0}^1 :  \mathscr{F}^1 \times \mathscr{G}^0 \to \mathscr{H}^1,   \]
such that \begin{equation}\label{compatibility}  d_{\mathscr{H}} (<x, y>_{0,0}^0)=<x, d_{\mathscr{G}}(y)>_{0,1}^1+<d_{\mathscr{F}}(x), y>_{1,0}^1 \end{equation}
for all $x\in \mathscr{F}^0$, $y \in \mathscr{G}^0$. Such a pairing is the same as a mapping
\[  \mathscr{F}^{\bullet} \otimes \mathscr{G}^{\bullet} \to \mathscr{H}^{\bullet}. \]

In our situation, we set 
\[ W_n\mathscr{F}^{\bullet}:= [ \jmath_*Z_1\Wno_{U}^r \xrightarrow{1-C} \jmath_*\Wno_{U}^r ] ;\]
\[ W_n\mathscr{G}^{\bullet}_{-m}:= [\Wno_{X|mD, \log}^{d+1-r}\to 0];   \]
\[ W_n\mathscr{H}^{\bullet}:=[\Wno_{X}^{d+1} \xrightarrow{1-C} \Wno_{X}^{d+1}]. \]
\begin{corollary}
We have a natural map of complexes 
\begin{equation}\label{pairingtwoterms}
 W_n\mathscr{F}^{\bullet} \to \varinjlim_m\sHom( W_n\mathscr{G}^{\bullet}_{-m}, W_n\mathscr{H}^{\bullet}). \end{equation}
\end{corollary}

\subsection{Ramified duality}
\begin{lemma}\label{affine}
	Let $D$ be a normal crossing divisor on $X$. Then the induced open immersion $\jmath: U:=X-|D| \to X$ is affine.
\end{lemma}
\begin{proof}
	To be an affine morphism is \'etale locally on the target, and \'etale locally $D$ is given by one equation.
\end{proof}
By Proposition \ref{logtocoh} (iii), in $D^b(X, \Zpn)$ we have 
\begin{equation}
R\jmath_*\Wno_{U,\log}^r \cong [ \jmath_*Z_1\Wno_{U}^r \xrightarrow{1-C} \jmath_*\Wno_{U}^r ]=W_n\mathscr{F}^{\bullet}.
\end{equation}

Then we have $\mathbb{H}^i(X, W_n\mathscr{F}^{\bullet})\cong \mathbb{H}^i(U, [Z_1W_n\Ome_U^r \to W_n\Ome_U^r])\cong H^i(U,\Wno_{U,\log}^r)$, by the above lemma and Proposition \ref{logtocoh} (iii). 
 Therefore the map (\ref{pairingtwoterms}) induces a map of cohomology groups, by taking hypercohomolgy, 
\begin{equation}\label{map.step1}
H^i(U, \Wno_{U,\log}^r) \to  \varinjlim\limits_{m} \mathbb{H}^{i}(X, \sHom( W_n\mathscr{G}^{\bullet}_{-m}, W_n\mathscr{H}^{\bullet}). 
\end{equation}

Note that we also have a natural evaluation map of two-term complexes:
\begin{equation}
\sHom(W_n\mathscr{G}^{\bullet}_{-m}, W_n\mathscr{H}^{\bullet}) \otimes W_n\mathscr{G}^{\bullet}_{-m} \to W_n\mathscr{H}^{\bullet}.
\end{equation}

As the construction of the pairing (\ref{pairingwithoutfiltration}) in \S3.1, this induces a cup product in hypercohomology of complexes
\begin{equation}
\mathbb{H}^{i}(X, \sHom( W_n\mathscr{G}^{\bullet}_{-m}, W_n\mathscr{H}^{\bullet})) \otimes \mathbb{H}^{d+2-i}_{X_s} (X, W_n\mathscr{G}^{\bullet}_{-m}) \to \mathbb{H}^{d+2}_{X_s} (X, W_n\mathscr{H}^{\bullet}).
\end{equation}

Therefore, it induces a map
\begin{equation}\label{map.step2}
 \mathbb{H}^{i}(X, \sHom( W_n\mathscr{G}^{\bullet}_{-m}, W_n\mathscr{H}^{\bullet})) \to \Hom_{\Zpn}(H^{d+2-i}_{X_s}(X, \Wno_{X|mD,\log}^{d+1-r}), H^{d+2}_{X_s}(X, \Wno_{X,\log}^{d+1}))
\end{equation} 

Now combining the maps (\ref{map.step1}), (\ref{map.step2}) and the trace map in Corollary \ref{tracemap} , we get the desired map
\begin{equation}\label{logarithmicduality}
H^i(U, \Wno_{U,\log}^r) \to  \varinjlim\limits_{m}\Hom_{\Zpn}( H^{d+2-i}_{X_s}(X, \Wno_{X|mD,\log}^{d+1-r}),\Zpn); 
\end{equation}

Now our main theorem in this paper is the following result.
\begin{theorem}\label{main theorem}
	For each $i, r\in \mathbb N$, the above map (\ref{logarithmicduality}) is an isomorphism. If we endow $ H^{d+2-i}_{X_s}(X,W_n\Ome_{X|mD,\log}^{d+1-r} )$ with the discrete topology, and endow $H^i(U, \Wno_{U,\log}^r)$ with the direct limit topology of compact-open groups, then the Pontryagin duality gives us an isomorphism of topological groups
	\[ \varprojlim\limits_m H^{d+2-i}_{X_s}(X,W_n\Ome_{X|mD,\log}^{d+1-r} ) \xrightarrow{\cong} Hom_{\Zpn,\mathrm{cont}}(H^i(U, \Wno_{U,\log}^r), \Zpn).   \]
\end{theorem}
In the rest of this section, we focus on the proof of this theorem.

\begin{proof}\renewcommand{\qedsymbol}{}
We can proceed analogously to the proof of Theorem \ref{unramifiedduality}.
First, we reduce this theorem to the case $n=1$, then using Cartier operator, we will study the relation between this theorem and coherent duality theorems.
\end{proof}	

	\textbf{Step 1: Reduction to the case $n=1$.} We have a commutative diagram:
\[\xymatrix@C=.2cm@R=.6cm{ \cdots \ar[r]& \scriptstyle H^i(U, W_{n-1}\Ome_{U,\log}^r) \ar[r] \ar[d] & \scriptstyle H^i(U, W_{n}\Ome_{U,log}^r) \ar[d]\ar[r] & \scriptstyle H^i(U,\Ome_{U,\log}^r) \ar[d] \ar[r]& \cdots\\
	\cdots \ar[r]& {\scriptstyle  \varinjlim\limits_m H^{d+2-i}_{X_s}(X, W_{n-1}\Ome_{X|mD,\log}^{d+1-r} )^{\vee}} \ar[r] & \scriptstyle \varinjlim\limits_m H^{d+2-i}_{X_s}(X, W_{n}\Ome_{X|mD,\log}^{d+1-r} )^{\vee} \ar[r] &\scriptstyle \varinjlim\limits_m H^{d+2-i}_{X_s}(X,\Ome_{X|mD,\log}^{d+1-r} )^{\vee} \ar[r] & \cdots
	}\]	
where $M^{\vee}$ means $\Hom_{\text{cont}}(M, \Q/\Z)$, i.e., its Pontryagin dual, for any locally compact topological abelian group $M$. The first row is the long exact sequence induced by the short exact sequence in Proposition \ref{logtocoh} (i), and the second row is an long exact sequence induced by Corollary \ref{exactofprosystem}. 

By the five lemma and induction, our problem is reduced  to the case $n=1$.

\textbf{Step 2: Proof of  the case $n=1$.} 
In this special case,  using the relation between the relative logarithmic de Rham-Witt  and coherent sheaves, we can reformulate our pairing.
\begin{theorem}(\cite[Thm. 1.2.1]{jszduality})
We have the following exact sequence
\[ 0 \to \Ome_{X|mD,\log}^{d+1-r} \to \Ome_{X|mD}^{d+1-r} \xrightarrow{C^{-1}-1} \Ome_{X|mD}^{d+1-r}/B\Ome_{X|mD}^{d+1-r} \to 0, \]
where $\Ome_{X|mD}^j=\Ome_X^{j}(\log D) \otimes \mathcal{O}_X(-mD)$ is defined in (\ref{twistedcoh}).
\end{theorem}
\begin{proof}
This is again a local problem, and the local proof in \cite{jszduality} also works in our situation. The key ingredient is to show $\Ome_{X|mD,\log}^i=\Ome_{X,\log}^i \cap \Ome_{X|mD}^i$. This can be obtained by a refinement of \cite[Prop. 1]{katogalois}. 
\end{proof}

\begin{lemma}
For any $\underline{m}\in \mathbb{N}^s$, we denote $mD= {\underline{m}}D$ as in  (\ref{defnofD}),  and $$\Ome_{X|mD}^j=\Ome_X^{j}(\log D)(-mD)=\Ome_X^{j}(\log D) \otimes \mathcal{O}_X(-mD) $$ as before. The parings 
\begin{eqnarray*}
&& <\alpha, \beta>_{0,0}^0=\alpha \wedge \beta :  Z\Ome_{X|-mD}^r \times \Ome_{X|(m+1)D}^{d+1-r} \to \Ome_{X|D}^{d+1};  \\
&& <\alpha, \beta>_{0,1}^1=-C(\alpha \wedge \beta) :  Z\Ome_{X|-mD}^r  \times \Ome_{X|(m+1)D}^{d+1-r}/B\Ome_{X|(m+1)D}^{d+1-r} \to \Ome_{X|D}^{d+1};  \\
&& <\alpha, \beta>_{1,0}^1= \alpha \wedge \beta :  \Ome_{X|-mD}^r  \times \Ome_{X|(m+1)D }^{d+1-r} \to \Ome_{X|D}^{d+1}, 
\end{eqnarray*}

define a paring of (two-term) complexes \begin{equation} \label{level-s-pairing}
\mathscr{F}_m^{\bullet} \times \mathscr{G}^{\bullet}_{-m} \to \mathscr{H}^{\bullet} \end{equation}
with \begin{eqnarray*}
&&  \mathscr{F}_m^{\bullet}=(Z\Ome_{X|-mD}^r \xrightarrow{1-C} \Ome_{X|-mD}^r )\\
&&  \mathscr{G}^{\bullet}_{-m}=(\Ome_{X|(m+1)D}^{d+1-r} \xrightarrow{C^{-1}-1} \Ome_{X|(m+1)D}^{d+1-r}/B\Ome_{X|(m+1)D}^{d+1-r} )\\
&&  \mathscr{H}^{\bullet}=(\Ome_{X|D}^{d+1} \xrightarrow{1-C} \Ome_{X|D}^{d+1})=(\Ome_{X}^{d+1} \xrightarrow{1-C} \Ome_{X}^{d+1}).
\end{eqnarray*}
\end{lemma}
\begin{proof}
This is easy to verify.
\end{proof}
By taking hypercohomology, the pairing (\ref{level-s-pairing}) induces a pairing of hypercohomology groups:
\begin{equation} \label{level-s-hypercomology}
\mathbb{H}^i(X, \mathscr{F}_m^{\bullet}) \times \mathbb{H}^{d+2-i}_{X_s}(X, \mathscr{G}_{-m}^{\bullet}) \to \mathbb{H}^{d+2}_{X_s}(X, \mathscr{H}^{\bullet}) \cong H^{d+2}_{X_s}(X, \Ome_{X,\log}^{d+1}) \xrightarrow{\text{Tr}} \Zp1.
\end{equation}

Now by Corollary \ref{twistedomega} and \ref{twistedzomega}, we can show 

\begin{proposition}\label{onesideiso}
The pairing (\ref{level-s-hypercomology}) induces the following  isomorphism.

\begin{equation} \label{isom_finit}
\begin{array}{rcl} \mathbb{H}^i(X, \mathscr{F}_m^{\bullet}) &\xrightarrow{\cong}& \text{Hom}_{\Zp1}(\mathbb{H}^{d+2-i}_{X_s}(X, \mathscr{G}_{-m}^{\bullet}), \Zp1)\\
& \cong &\text{Hom}_{\Zp1}(H^{d+2-i}_{X_s}(X, \Wno_{X|(m+1)D,\log}^{d+1-r}), \Zp1)  .
\end{array}
\end{equation}

\end{proposition}
\begin{proof}
This can be done by using the hypercohomology spectral sequences
\[  ^IE_1^{p,q}= H^q(X, \mathscr{F}_m^p)  \Rightarrow \mathbb{H}^{p+q}(X, \mathscr{F}_m^{\bullet}) ; \]
\[  ^{II}E_1^{p,q}= H_{X_s}^q(X, \mathscr{G}_{-m}^p)  \Rightarrow \mathbb{H}_{X_s}^{p+q}(X, \mathscr{G}_{-m}^{\bullet}). \]
 Corollary \ref{twistedomega} and \ref{twistedzomega} tell us that \[  ^IE_1^{p,q} \cong \text{Hom}_{\Zpn}( ^{II}E_1^{d+1-q, 1-p}, \Zpn),   \]
and this isomorphism is compatible with $d_1$. Hence we still have this kind of duality at the $E_2$-pages. Note that, by definition, $p\neq 0,1$, $^IE_1^{p,q}= ^{II}E_1^{p,q}=0$. Hence both spectral sequences degenerate at the $E_2$-pages. Therefore we have the isomorphism in the claim. 
\end{proof}
Up to now, we haven't used any topological structure on the (hyper-)cohomology group.  Note that $H^{d+2-i}_{X_s}(X, \Wno_{X|(m+1)D,\log}^{d+1-r})$ is not a finite group in general, and we endow it with discrete topology. Hence we endow $\mathbb{H}^i(X, \mathscr{F}_m^{\bullet}) $ with the compact-open topology via the isomorphism \eqref{isom_finit}.  Now the Pontryagin duality theorem implies:
\begin{proposition}
There is a perfect pairing of topological $\Zp1$-modules:
\begin{equation}
\varinjlim_m\mathbb{H}^i(X, \mathscr{F}_m^{\bullet}) \times \varprojlim_m H^{d+2-i}_{X_s}(X, \Omega_{X|(m+1)D,\log}^{d+1-r}) \to \Zp1
\end{equation}
where the first term is endowed with direct limit topology, and the second with the inverse limit topology.
\end{proposition}
\begin{proof}
Note that the Pontryagin dual $\text{Hom}_{\text{cont}}(\cdot, \Zp1)$ commutes with direct and inverse limits. Then the proof is straightforward.
\end{proof}

\begin{remark}In case that the groups $H^{d+2-i}_{X_s}(X, \Omega_{X|(m+1)D,\log}^{d+1-r})$ are finite for all $m$(e.g.,$d=0$ or $r=0$), then
the direct limit topology of finite (compact-open) topological spaces is discrete, and the inverse limit topololgy of finite discrete topological spaces is profinite. 
\end{remark}

We still need to calculate the direct limit term in the above proposition.
\begin{proposition}\label{directisU}
The direct limit $\varinjlim\limits_m\mathbb{H}^i(X, \mathscr{F}_m^{\bullet}) \cong H^i(U, \Ome^r_{U,\log})$.
\end{proposition}
\begin{proof}
First, direct limits commute with (hyper-)cohomology, hence \[ \varinjlim_m\mathbb{H}^i(X, \mathscr{F}_m^{\bullet}) =\mathbb{H}^i(X,  \varinjlim_m\mathscr{F}_m^{\bullet}). \]
Note that $$ \varinjlim_m\mathscr{F}_m^{\bullet}= [\jmath_*Z\Ome_U^r \xrightarrow{1-C} \jmath_*\Ome_{U}^r]. $$ 
For coherent sheaves, the affine morphism $\jmath$ (see Lemma \ref{affine}) gives an exact functor $\jmath_*$. Hence we have
\[ \mathbb{H}^i(X, [\jmath_*Z\Ome_U^r \xrightarrow{1-C} \jmath_*\Ome_{U}^r])= \mathbb{H}^i(U, [Z\Ome_U^r \xrightarrow{1-C} \Ome_{U}^r]) =H^i(U, \Ome_{U,\log}^r),  \]
where the last equality follows from the special case $n=1$ of Proposition \ref{logtocoh} (iii).
\end{proof}

\begin{proof}[Proof of \textbf{Step 2}]
The duality theorem in case $n=1$ directly follows from the above two propositions.
\end{proof}
Now the proof of our main Theorem \ref{main theorem} is complete.

We denote 
\begin{equation*}
\begin{array}{lrl}
\Phi : H^i(U, \Wno_{U,\log}^r) &\xrightarrow{\cong}& \text{Hom}_{\text{cont}}( \varprojlim\limits_{m} H^{d+2-i}_{X_s}(X, \Wno_{X|mD,\log}^{d+1-r}),\Zpn);\\
&\cong& \varinjlim\limits_m\text{Hom}( H^{d+2-i}_{X_s}(X, \Wno_{X|mD,\log}^{d+1-r}),\Zpn)
\end{array}
\end{equation*}

\[ \Psi : \varprojlim\limits_{m} H^{d+2-i}_{X_s}(X, \Wno_{X|mD,\log}^{d+1-r}) \xrightarrow{\cong} \text{Hom}_{\rm{cont}}(H^i(U, \Wno_{U,\log}^r),\Zpn). \]

Using this duality theorem,  we may at last define a filtration as follows:
\begin{definition}\label{newfiltration}
Assume $X, X_s, D, U$ are as before. For any $\chi \in H^i(U, \Wno_{U,\log}^r)$, we define the higher Artin conductor 
\[  \text{ar}(\chi):=\text{min} \{ m\in \mathbb{N}_0^s \ | \ \Phi(\chi)\    \text{factors through }   H^{d+2-i}_{X_s}(X, \Wno_{X|mD,\log}^{d+1-r})  \},   \]
For $m\in \mathbb{N}^s$, we define 
\[ \text{Fil}_mH^i(U, \Wno_{U,\log}^r):=\{ \chi \in H^i(U, \Wno_{U,\log}^r) |\  \text{ar}(\chi) \leq m \},  \]
and
\[ \pi^{\text{ab}}_1(X,mD)/p^n:= \text{Hom}(\text{Fil}_mH^1(U,\Zpn),\Q/\Z)\]
endowed with the usual profinite topology of the dual.
\end{definition}

It is clear that $\text{Fil}_{\bullet}$ is an increasing filtration with respect to the semi-order on $\mathbb{N}^s$. We have 
 \[ \text{Fil}_mH^i(U, \Wno_{U,\log}^r)= \mathrm{Image}(\text{Hom}_{\Zpn}( H^{d+2-i}_{X_s}(X, \Wno_{X|mD,\log}^{d+1-r}), \Zpn)\to H^i(U, \Wno_{U,\log}^r)) .
  \] 
The quotient  $\pi^{\text{ab}}_1(X,mD)/p^n$ can be thought of as classifying abelian \'etale coverings of $U$ whose degree divides $p^n$ with ramification bounded by the divisor $mD$.

\section{Comparison with the classical case}
In this section, we want to compare our filtration with the classical one in the local ramification theory.

\subsection{Local ramification theory}
Let $K$ be a local field, i.e., a complete discrete valuation field of characteristic $p>0$, let $\mathcal{O}_K$ be its ring of integers,  let $k$ be its finite residue field, and let $\nu_K$ be its valuation. We fix a uniformizer $\pi \in \mathcal{O}_K$, which generates the maximal ideal $\m\in \mathcal{O}_K$. 

The local class field theory \cite{serrelocalfield} gives us an Artin reciprocity homomorphism
\[ \text{Art}_K:  K^{\times} \to Gal(K^{\text{ab}}/K), \]
where $K^{\text{ab}}$ is the maximal abelian extension of $K$. Note that both $K^{\times}$ and $Gal(K^{\text{ab}}/K)$ are topological groups. Recall the topological structure on $K^{\times}$ is given by the valuation on $K$, and $Gal(K^{\text{ab}}/K)$ is the natural profinite topology.

For any $m \in \mathbb{N}$, the Atrin map induces an isomorphism of topological groups
\[ \text{Art}_K\otimes 1:  K^{\times}\otimes \Z/m\Z \xrightarrow{\cong} Gal(K^{\text{ab}}/K)\otimes \Z/m\Z . \]
In particular, take $m=p^n$, it gives:
\begin{equation}\label{localclassfield}
\text{Art}_K\otimes 1:  K^{\times}/(K^{\times})^{p^n} \xrightarrow{\cong} Gal(K^{\text{ab}}/K)\otimes \Zpn . \end{equation}

For $n \geq 1$, the Artin-Schreier-Witt theory tells us there is  a natural isomorphism
\begin{equation}\label{artinschreier}
\delta_n: W_n(K)/(1-F)W_n(K) \xrightarrow{\cong} H^1(K, \Zpn), 
\end{equation}
where $W_n(K)$ is the ring of Witt vector of length $n$ and $F$ is the Frobenius. 

Note that $H^1(K, \Zpn)$ is dual to $Gal(K^{\text{ab}}/K)\otimes \Zpn$, the interplay between (\ref{localclassfield}) and (\ref{artinschreier}) gives rise to the following theorem.
\begin{theorem}[Artin-Schreier-Witt]
	There is a perfect pairing of topological groups, that we call the Artin-Schreier-Witt symbol
	\begin{align}\label{artinschreierwitt}
	W_n(K)/(1-F)W_n(K) \times K^{\times}/(K^{\times})^{p^n} &\longrightarrow \Zpn \\
	(a,b) &\mapsto [a,b):=(b,L/K)(\alpha)-\alpha \nonumber
	\end{align}
	where $(1-F)(\alpha)=a$, for some $\alpha \in W_n(K^{\text{sep}})$, $L=K(\alpha)$, $(b,L/K)$ is the norm residue of $b$ in $L/K$, and the topological structure on the first term is discrete, on the second term is induced from $K^{\times}$.
	\end{theorem} 
\begin{proof}
	This pairing is non-degenerate \cite[Prop. 3.2]{thomas}. Taking the topological structure into account, we get the perfectness by Pontryagin duality.
\end{proof}
We have filtrations on the two left terms in the pairing (\ref{artinschreierwitt}). On $W_n(K)$, Brylinski \cite{brylinski} and Kato \cite{katoswanconductor} defined an increasing filtration,  called the Brylinski-Kato filtration, using the valuation on $K$:

\[\text{fil}^{\log}_mW_n(K)=\{ (a_{n-1},\cdots, a_1, a_0) \in W_n(K)|\ p^i\nu_K(a_i) \geq -m\}. \]

We also have its non-log version introduced by Matsuda \cite{matsuda}.
\begin{equation}
\text{fil}_mW_n(K)=\text{fil}^{\log}_{m-1}W_n(K)+V^{n-n'}\text{fil}_{m}^{\log}W_{n'}(K),
\end{equation}
where $n'=\text{min}\{n, \text{ord}_p(m) \}$ and $V: W_{n-1}(K) \to W_n(K)$ is the Verscheibung on Witt vectors.

Both of them induce filtrations on the quotient $W_n(K)/(1-F)W_n(K)$, and we define
\begin{equation}\text{fil}^{\log}_mH^1(K,\Zpn)=\delta_n(\text{fil}^{\log}_m(W_n(K)/(1-F)W_n(K)))=\delta_n(\text{fil}^{\log}_mW_n(K)), \end{equation}
\begin{equation}\text{fil}_mH^1(K,\Zpn)=\delta_n(\text{fil}_m(W_n(K)/(1-F)W_n(K)))=\delta_n(\text{fil}_mW_n(K)). \end{equation}

We have the following fact on the relation of two above filtrations.
\begin{lemma}(\cite{katoswanconductor}, \cite{matsuda})\label{logvsnonlog} For an integer $m\geq 1$, we have
\begin{itemize}
	\item[(i)] $\text{fil}_mH^1(K,\Zpn) \subset \text{fil}^{\log}_{m}H^1(K,\Zpn) \subset \text{fil}_{m+1}H^1(K,\Zpn)$,
	\item[(ii)] $\text{fil}_{m}H^1(K,\Zpn)=\text{fil}_{m-1}^{\log}H^1(K,\Zpn)$ if $(m,p)=1$. 
\end{itemize}
\end{lemma}
\begin{remark}
 The non-log version of Brylinski-Kato filtration is closely related to the K\"ahler differential module $\Ome_K^1$ \cite{matsuda}, and it has an higher analogy on $H^1(U)$, where $U$ is an open smooth subscheme of a normal variety $X$ over a perfect field with $(X-U)_{\text{red}}$ is the support of an effective Cartier divisor \cite{kerzsaito}. 
\end{remark}
On $K^{\times}$, we have a natural decreasing filtration given by:
\[U_K^{-1}=K^{\times}, U_K^0=\mathcal{O}_K^{\times}, \   U_K^m=\{x\in \mathcal{O}_K^{\times}|\  x \equiv 1 \ \text{mod}\  \pi^m\}.\]

The following theorem says the paring (\ref{artinschreierwitt}) is compatible with these filtrations. 
\begin{theorem}(\cite[Thm. 1]{brylinski})
	Underling the Artin-Schreier-Witt symbol (\ref{artinschreierwitt}), the orthogonal complement of $\text{fil}_{m-1}^{\log}H^1(K,\Zpn)$ is $U_K^m \cdot (K^{\times})^{p^n}/(K^{\times})^{p^n}$, for any integer $m \geq 1$.
\end{theorem}
\begin{proof}
	A more complete proof can be found in \cite[\S 5]{thomas}.
\end{proof}
\begin{corollary}\label{localcompatibletheorem}
The Artin-Schreier-Witt symbol (\ref{artinschreierwitt}) induces a perfect pairing of finite groups
\[ \text{fil}_{m}(W_n(K)/(1-F)W_n(K)) \times K^{\times}/(K^{\times})^{p^n} \cdot U_K^m  \longrightarrow \Zpn.\]
\end{corollary}
\begin{proof}
First, note that the filtration  $\{U^m_K\}_{m}$ has no jump greater or equal to 0 that divisible by $p$, as the residue field of $K$ is perfect. Then, we may assume $(m,p)=1$. By Lemma \ref{logvsnonlog} (ii) and the above Brylinski's theorem, we have, the orthogonal complement of  $\text{fil}_{m}H^1(K,\Zpn)=\text{fil}_{m-1}^{\log}H^1(K,\Zpn)$ is $U_K^m \cdot (K^{\times})^{p^n}/(K^{\times})^{p^n}$. The rest follows easily from the fact that the Pontryagin dual $H^{\wedge}$ of an open subgroup of a locally compact group $G$ is isomorphic to $G^{\wedge}/H^{\perp}$, where $H^{\perp}$ is the orthogonal complement of $H$.
\end{proof}

\subsection{Comparison of filtrations}
Let $X=B=\text{Spec}\Fp[[t]]$, $D=s=(t)$ be the unique closed point. Then $U=\text{Spec}(\Fp((t)))$. Our duality theorem \ref{main theorem} in this setting is:
\begin{corollary} \label{localfieldduality}
The pairing 
\[H^i(K, \Wno_{U,\log}^j) \times \varprojlim\limits_mH^{2-i}_{s}(B, \Wno_{B|mD,\log}^{1-j}) \to \Zpn \]
is a perfect paring of topological groups.
\end{corollary}
In particular, we take $i=1, j=0$, and get
\begin{equation}\label{pairingfromdualitylocalfield}
H^1(U, \Zpn) \times \varprojlim\limits_mH^1_s(B, \Wno_{B|mD,\log}^{1}) \to \Zpn.
\end{equation}

We want to compare this pairing (\ref{pairingfromdualitylocalfield}) with the Artin-Schreier-Witt symbol (\ref{artinschreierwitt}). 
\begin{lemma}\label{comparelimit}
We have
$ H^{1}_s(B, \Wno_{B|mD,\log}^1) \cong K^{\times}/(K^{\times})^{p^n}\cdot U^m_K  $. The diagram 
\[ \xymatrix{ H^1_s(B, \Wno_{B|(m+1)D,\log}^{1}) \ar[d]  &  K^{\times}/(K^{\times})^{p^n}\cdot U^{m+1}_K \ar[d]\ar[l]_-{\cong}\\
H^{1}_s(B,\Wno_{B|mD,\log}^1)  &K^{\times}/(K^{\times})^{p^n}\cdot U^m_K \ar[l]_-{\cong} }\]
commutes, where the left vertical arrow is induced by the morphism of sheaves, and the right vertical arrow is given by projection.
 In particular, 
\[  \varprojlim\limits_mH^1_s(B, \Wno_{B|mD,\log}^{1}) \cong K^{\times}/(K^{\times})^{p^n}. \]
\end{lemma}
\begin{proof}
We prove this by induction on $n$. If $n=1$, the localization sequence gives the following exact sequence
\[ 0 \to H^0_s(B, \Ome_{B|mD, \log}^1) \to H^0(B, \Ome_{B|mD,\log}^1) \to H^0(U,\Ome_{U,\log}^1) \to H^1_s(B, \Ome_{B|mD,\log}^1) \to 0.  \]
The Bloch-Gabber-Kato theorem \cite{blochkato} says $K^{\times}/(K^{\times})^{p}\xrightarrow{\cong} H^0(U,\Ome_{U,\log}^1) $, and by definition, it is easy to see that $U_K^m\cdot (K^{\times})^p/(K^{\times})^{p}\xrightarrow{\cong}H^0(B, \Ome_{B|mD,\log}^1) $. For the induction process, we use the exact sequence in Theorem (\ref{twistedexact}):
\[ 0 \to W_{n-1}\Ome_{B|[m/p]D,\log}^{1} \xrightarrow{p} \Wno_{B|mD,\log}^1 \xrightarrow{R} \Ome_{X|mD,\log}^1 \to 0 .\]
Note that the first term involves dividing by $p$. But for the filtration  $\{U^m_K\}_{m}$, there are no jump greater or equal to 0 that divisible by $p$, as the residue field of $K$ is perfect. The commutativity of the diagram follows also directly from the above computation.
\end{proof}

Now our main result in this section is the following:
\begin{proposition}\label{filagree} The filtration we defined in Definition \ref{newfiltration} is same as the non-log version of  Brylinski-Kato filtration, i.e., for any integer $m\geq 1$,
\[\text{Fil}_mH^1(U, \Zpn)=\text{fil}_{m}H^1(U, \Zpn).\]
\end{proposition}
\begin{proof}
	We have the following isomorphisms
	 \begin{equation*}
	 \begin{array}{rcl}
	 Fil_mH^1(U, \Zpn)&=&\text{Hom}_{\Zpn}(H^1_s(B, \Ome_{B|mD,\log}^1), \Zpn)\\
	 &=&\text{Hom}_{\Zpn}(K^{\times}/(K^{\times})^{p^n}\cdot U^m_K, \Zpn)\\
	 &=&fil_{m}H^1(U, \Zpn)
	 \end{array}
	 \end{equation*} 
where the first identification is due to the fact that the transition maps are surjective, the second equality is given by Lemma \ref{comparelimit}, and the last is Corollary \ref{localcompatibletheorem}.
\end{proof}

\providecommand{\bysame}{\leavevmode\hbox to3em{\hrulefill}\thinspace}
\providecommand{\MR}{\relax\ifhmode\unskip\space\fi MR }
\providecommand{\MRhref}[2]{%
	\href{http://www.ams.org/mathscinet-getitem?mr=#1}{#2}
}
\providecommand{\href}[2]{#2}


\end{document}